\documentclass[a4paper,leqno,10pt]{amsart}

\usepackage{epsf,graphicx,epsfig}
\usepackage{amscd}
\usepackage{amsmath,latexsym,amssymb,amsthm}
\usepackage[nospace,noadjust]{cite}
\usepackage{textcomp}
\usepackage{setspace,cite}
\usepackage{lscape,fancyhdr,fancybox}
\usepackage{stmaryrd}
\usepackage[all,cmtip]{xy}
\usepackage{tikz}
\usepackage{cancel}
\usetikzlibrary{shapes,arrows,decorations.markings}
\usepackage{graphicx}

\usepackage{color}
\usepackage{url}
\usepackage{enumerate}
\usepackage[mathscr]{euscript}
  \theoremstyle{plain}

\swapnumbers
    \newtheorem{theorem}{Theorem}[section]
    \newtheorem{proposition}[theorem]{Proposition}
     
   \newtheorem{lemma}[theorem]{Lemma}

    \newtheorem{subsec}[theorem]{}
\theoremstyle{definition}
    \newtheorem{definition}[theorem]{Definition}
        \newtheorem{remark}[theorem]{Remark}
\theoremstyle{remark}

\title{}
\author{}
\date{}
\usepackage{amssymb}

\usepackage{hyperref}
\hypersetup{
	colorlinks,
	citecolor=blue,
	filecolor=black,
	linkcolor=blue,
	urlcolor=black
}

\begin{document}
\title{A cohomological study of modified Rota-Baxter algebras}

\author{Apurba Das}
\address{Department of Mathematics,
Indian Institute of Technology Kharagpur, Kharagpur-721302, West Bengal, India.}
\email{apurbadas348@gmail.com, apurbadas348@maths.iitkgp.ac.in}

%\author{Nishant Rathee}
%\address{School of Mathematics, Harish-Chandra Research Institute, HBNI, Chhatnag Road, Jhunsi, Allahabad-211019, India.}
%\email{nishant@hri.res.in}

%\begin{center}

%\end{center}
%

%\author{Samir Kumar Hazra}
%\address{\textcolor{red}{LEFT}}
%\email{\textcolor{red}{LEFT}}

%\author{Satyendra Kumar Mishra}
%

%\curraddr{}
%\email{}


\begin{abstract}
A modified Rota-Baxter algebra is an algebra equipped with an operator that satisfies the modified Yang-Baxter equation. In this paper, we define the cohomology of a modified Rota-Baxter algebra with coefficients in a suitable bimodule. We relate our cohomology of a modified Rota-Baxter algebra with the known cohomology theory of a Rota-Baxter algebra. As applications of our cohomology, we study formal one-parameter deformations and abelian extensions of modified Rota-Baxter algebras. \\
\end{abstract}

%{17B38, 16D20, 16E40, 18N40}

\maketitle

%\noindent {\bf Abstract:} \\

%\qquad 2020 Mathematics Subject Classification: {17B38, 16D20, 16E40, 18N40.}

\noindent

\thispagestyle{empty}

\tableofcontents

\vspace{0.2cm}

\section{Introduction}
Rota-Baxter operators pay very much attention in the last twenty years as they have many applications in mathematics and physics. Rota-Baxter operators originated in a work of Baxter in the fluctuation theory of probability \cite{baxter}. Subsequently, many well-known mathematicsians, such as Atkinson, Cartier and Rota studied Rota-Baxter operators from combinatorial points of view \cite{atkinson,cartier,rota}. See \cite{aguiar,aguiar-pre-lie,bai,connes,guo-book,guo-keigher} for the appearence of Rota-Baxter operators in combinatorics, splitting of algebras, infinitesimal bialgebras and renormalizations in quantum field theory. Let $A$ be an (associative) algebra over a field ${\bf k}$ of characteristic $0$. A linear map $P: A \rightarrow A$ is said to be a {\bf Rota-Baxter operator} of weight $\lambda \in {\bf k}$ if the map $P$ satisfies
\begin{align*}
P(a) \cdot  P(b) = P \big( P(a) \cdot  b + a \cdot  P(b) \big) + \lambda ~ \! P (a \cdot  b), \text{ for all } a, b \in A.
\end{align*}
Here $~\cdot ~$ denotes the multiplication map in $A$. A pair $(A, P)$ consisting of an algebra $A$ and a Rota-Baxter operator $P$ of weight $\lambda$ is said to be a {\bf Rota-Baxter algebra} of weight $\lambda$. Recall that \cite{GuoLin} a bimodule over a Rota-Baxter algebra $(A, P)$ of weight $\lambda$ consists of a pair $(M, Q)$ in which $M$ is an $A$-bimodule and $Q: M \rightarrow M$ is a linear map satisfying 
\begin{align*}
P(a) \cdot  Q(u) =~& Q \big( P(a) \cdot  u + a \cdot  Q(u) \big) + \lambda ~ \! Q (a \cdot  u),\\
Q(u) \cdot  P(a) =~& Q \big( Q(u) \cdot  a + u \cdot  P(a) \big) + \lambda ~ \! Q (u \cdot a),
\end{align*}
for all $a \in A$ and $u \in M$. Here $~ \cdot ~$ also denotes the left and right $A$-actions on $M$. See also \cite{wang-zhou} for details.

\medskip

Rota-Baxter operators can also be defined on other algebraic structures, such as Lie algebras, Poisson algebras etc. In \cite{semenov} Semenov-Tyan-Shanski\u{i} observed that under certain conditions, a Rota-Baxter operator of weight $0$ on a Lie algebra is simply the operator form of the classical Yang-Baxter equation (CYBE). In the same paper, the author also introduced a closely related notion of modified classical Yang-Baxter equation (modified CYBE) whose solutions are called modified $r$-matrices. The modified CYBE has found important applications in the study of Lax equations, affine geometry on Lie groups, factorization problems in Lie algebras and compatible Poisson structures \cite{bai-lax,bor,li,sza}. The modified CYBE can be obtained from the CYBE by a suitable transformation, however they plays independent role in mathematical physics. This led various authors to consider them independently. Motivated by such study, the authors in \cite{zhang,zhang2,zhang3} considered the associative analogue of the modified CYBE, called the modified associative Yang-Baxter equation of weight $\kappa \in {\bf k}$ (in short modified AYBE of weight $\kappa$). A solution of the modified AYBE of weight $\kappa$ is called a modified Rota-Baxter operator of weight $\kappa$. Thus, modified Rota-Baxter operators are the associative analogue of modified $r$-matrices. Given an associative algebra $A$, a linear map $R: A \rightarrow A$ is a {\bf modified Rota-Baxter operator} of weight $\kappa \in {\bf k}$ if
\begin{align*}
R(a) \cdot  R(b) = R \big( R(a) \cdot  b + a \cdot  R(b) \big) + \kappa ~ \! a \cdot  b, ~ \text{ for all } a, b \in A.
\end{align*}
An associative algebra equipped with a modified Rota-Baxter operator of weight $\kappa$ is called a modified Rota-Baxter algebra of weight $\kappa$. It has been observed \cite{zhang,zhang2} that if $(A, P)$ is a Rota-Baxter algebra of weight $\lambda$, then $(A, R = \lambda ~ \! \mathrm{id} + 2P)$ is a modified Rota-Baxter algebra of weight $\kappa = - \lambda^2$.

\medskip

Algebraic structures are better understood by their invariants, such as homology, cohomology and $K$-groups. All these invariants have their own significance. The cohomology theory of an algebraic structure govern formal deformations of the structure and classify its abelian extensions. For example, the classical Hochschild cohomology theory is the cohomology for associative algebras \cite{hoch}. To know the better description of the cohomology, one need to understand representations of the given algebraic structure. In this paper, we first consider bimodules over modified Rota-Baxter algebras and construct the corresponding semidirect product. A bimodule over a modified Rota-Baxter algebra $(A,R)$ consists of a pair $(M,S)$ in which $M$ is an $A$-bimodule and $S: M \rightarrow M$ is a linear map satisfying some compatibilities. It follows that any modified Rota-Baxter algebra $(A, R)$ is a bimodule over itself. Given a modified Rota-Baxter algebra $(A,R)$, we show that the underlying space $A$ carries a new associative algebra structure with the product $*_R$ given by
\begin{align*}
a *_R b = R(a) \cdot  b + a \cdot  R(b), ~ \text{ for } a, b \in A.
\end{align*}
We denote this associative algebra by $A_R$. Further, if $(M,S)$ is a bimodule over the modified Rota-Baxter algebra $(A,R)$, then there is an $A_R$-bimodule structure on $M$ (denoted by $\widetilde{M}$). Thus, we may define the Hochschild cochain complex of the algebra $A_R$ with coefficients in the $A_R$-bimodule $\widetilde{M}$. On the other hand, we can consider the Hochschild cochain complex of the given algebra $A$ with coefficients in the $A$-bimodule $M$. As a biproduct of the above two cochain complexes, we construct a new complex defining the cohomology of the modified Rota-Baxter algebra $(A, R)$ with coefficients in the bimodule $(M,S)$. We compare our cohomology of modified Rota-Baxter algebras with the cohomology of Rota-Baxter algebras recently introduced in \cite{DasSK,wang-zhou}. More precisely, we show that the cohomology of a Rota-Baxter algebra $(A,P)$ is isomorphic to the cohomology of the modified Rota-Baxter algebra $(A, R= \lambda ~ \! \mathrm{id} + 2P)$.

\medskip

The formal deformation theory of some algebraic structure began with the seminal work of Gerstenhaber for associative algebras \cite{gers,gers-ring}. Subsequently, the theory has been generalized by Nijenhuis and Richardson for Lie algebras \cite{nij-ric}. See \cite{bala,doubek,balav,gers-sch} for more details about deformations of various algebras and morphisms between them. Recently, various authors have considered formal deformations of Rota-Baxter operators, Rota-Baxter algebras and modified $r$-matrices motivated by the deformations of classical $r$-matrices \cite{das-rota,DasSK,jiang-sheng,laza-rota,tang,wang-zhou,modified-sheng}. In this paper, we consider the formal deformation theory of a modified Rota-Baxter algebra $(A,R)$. In this theory, we simultaneously deform the algebra $A$ and the modified Rota-Baxter operator $R$. We show that the infinitesimal in a formal deformation of $(A,R)$ is a $2$-cocycle in the cohomology complex of $(A,R)$ with coefficients in itself. Moreover, the infinitesimals corresponding to equivalent deformations are cohomologous, hence they corresponds to the same element in the second cohomology group. We also discuss the rigidity of a modified Rota-Baxter algebra $(A,R)$ and show that the vanishing of the second cohomology group is a sufficient condition for the rigidity of $(A,R)$.

\medskip

In the next, we define and study abelian extensions of a modified Rota-Baxter algebra $(A, R)$ by a bimodule $(M,S)$ over it. We show that the isomorphism classes of such abelian extensions can be classified by the second cohomology of the modified Rota-Baxter algebra $(A,R)$ with coefficients in the bimodule $(M,S).$

\medskip

The paper is organized as follows. In Section \ref{sec-2}, we recall some basic definitions about modified Rota-Baxter algebras, and define bimodules over them. The cohomology of a modified Rota-Baxter algebra with coefficients in a bimodule is defined in Section \ref{sec-3}. We also compare this cohomology with the cohomology of Rota-Baxter algebras. In Section \ref{sec-4}, we study formal deformation theory and rigidity of modified Rota-Baxter algebras. Finally, abelian extensions are considered in Section \ref{sec-5}.

\medskip

All vector spaces, linear and multilinear maps, tensor products are over a field ${\bf k}$ of characteristic $0$.

%*********************************************************

%***************************************************************************************

\section{Modified Rota-Baxter algebras}\label{sec-2}

In this section, we recall modified Rota-Baxter algebras and define bimodules over them. Given a modified Rota-Baxter algebra and a bimodule over it, we also consider the corresponding semidirect product. We show that a modified Rota-Baxter algebra induces a new associative structure on the underlying vector space. We also consider some bimodules over this induced associative algebra.

\begin{definition}
\begin{enumerate}
\item Let $A$ be an associative algebra and $\kappa \in {\bf k}$. A linear map $R : A \rightarrow A$ is said to be a {\bf modified Rota-Baxter operator of weight $\kappa$} if
\begin{align*}
R(a) \cdot  R(b) = R \big( R(a) \cdot  b + a \cdot  R(b) \big) + \kappa ~ \! a \cdot  b, ~\text{ for } a, b \in A.
%R(a) \cdot R(b) = R \big( R(a) \cdot b + a \cdot R(b) \big) + \kappa~ a \cdot b, ~\text{ for } a, b \in A.
\end{align*}
\item A {\bf modified Rota-Baxter algebra of weight $\kappa$} is a pair $(A, R)$ consisting of an associative algebra $A$ and a modified Rota-Baxter operator $R: A \rightarrow A$ of weight $\kappa$.
\end{enumerate}
\end{definition}

\begin{definition}
Let $(A, R)$ and $(A', R')$ be two modified Rota-Baxter algebras of weight $\kappa$. A {\bf morphism} of modified Rota-Baxter algebras from $(A,R)$ to $(A', R')$ is an algebra homomorphism $\varphi : A \rightarrow A'$ that satisfies $R' \circ \varphi = \varphi \circ R$.
\end{definition}

%\begin{example}
%\textcolor{red}{Atleat 3 nice examples}
%\end{example}

There is a close relationship between Rota-Baxter algebras and modified Rota-Baxter algebras. 

\begin{lemma}
Let $A$ be an associative algebra. Then a linear map $P: A \rightarrow A$ is a Rota-Baxter operator of weight $\lambda \in {\bf k}$ if and only if the map $R = \lambda ~ \! \mathrm{id} + 2P$ is a modified Rota-Baxter operator of weight $\kappa = - \lambda^2.$ Equivalently, $(A, P)$ is a Rota-Baxter algebra of weight $\lambda$ if and only if $(A, R= \lambda ~ \! \mathrm{id} + 2P)$ is a modified Rota-Baxter algebra of weight $\kappa = - \lambda^2.$
\end{lemma}

Let $A$ be an associative algebra. Consider the direct sum $A \oplus A$ with the direct product associative structure
\begin{align*}
(a, b) \bullet (a', b') = (a \cdot a', b \cdot b'), \text{ for } (a, b), (a', b') \in A \oplus A.
\end{align*}
Note that the subspace $A_\triangle = \{ (a, a) | ~ a \in A \} \subset A \oplus A$ is an associative subalgebra. Modified Rota-Baxter operators of weight $-1$ are useful to find compliments of $A_\triangle$ which is also a subalgebra. To see this, let $A_{- \triangle} \subset A \oplus A$ be the subspace given by $A_{- \triangle} = \{ (a, -a)|~ a \in A \}$. For any linear map $R : A \rightarrow A$, we define a map $\widehat{R} : A_{- \triangle} \rightarrow A_{\triangle}$ by $\widehat{R}(a, -a) = (- R(a), -R(a))$, for $(a, -a) \in A_{- \triangle}$. Consider the graph 
\begin{align*}
\mathrm{Gr}(\widehat{R}) =~& \{ \widehat{R}(a, -a) + (a, -a) |~ a \in A \}\\
=~& \{ (-R(a) + a , -R(a) -a)|~ a \in A \} \subset A \oplus A
\end{align*}
of the map $\widehat{R}$. Note that $\mathrm{Gr}(\widehat{R})$ is a compliment (as a subspace) of $A_\triangle$ in the space $A \oplus A$. However, $\mathrm{Gr}(\widehat{R})$ is in general not a subalgebra. For any $a, b \in A$, we observe that
\begin{align*}
&\big(  -R(a) + a, -R(a) -a \big)  \bullet \big( -R(b) + b, -R(b) -b  \big) \\
&= \big( R(a) \cdot R(b) - R(a) \cdot b - a \cdot R(b) + a \cdot b,~ R(a) \cdot R(b) + R(a) \cdot b + a \cdot R(b) + a \cdot b    \big)
\end{align*}
This is in $\mathrm{Gr}(\widehat{R})$ if and only if
\begin{align*}
R(a) \cdot R(b) + a \cdot b = R \big(  R(a) \cdot b + a \cdot R(b) \big).
\end{align*}
In other words, $\mathrm{Gr}(\widehat{R})$ is a subalgebra of $A \oplus A$ if and only if $R$ is a modified Rota-Baxter operator of weight $-1$.
\medskip

Next we define bimodules over a modified Rota-Baxter algebra.

\begin{definition}
Let $(A, R)$ be a modified Rota-Baxter algebra of weight $\kappa$. A {\bf bimodule} over $(A, R)$ consists of a pair $(M,S)$ in which $M$ is an $A$-bimodule and $S:M \rightarrow M$ is a linear map satisfying for $a \in A,$ $u \in M$,
\begin{align*}
R(a) \cdot S(u) = S \big( R(a) \cdot u + a \cdot S(u) \big) + \kappa ~ \! a \cdot u,\\
S(u) \cdot R(a) = S \big( S(u) \cdot a + u \cdot R(a) \big) + \kappa ~ \! u \cdot a.
\end{align*}
\end{definition}

If $(A,R)$ is a modified Rota-Baxter algebra of weight $\kappa$, then the pair $(A,R)$ itself is a bimodule over it, where the $A$-bimodule structure on $A$ is given by the algebra multiplication. This is called the adjoint bimodule.

The following result finds the relation between bimodules over Rota-Baxter algebras and over modified Rota-Baxter algebras.

\begin{proposition}
Let $(A, P)$ be a Rota-Baxter algebra of weight $\lambda$, and $(M, Q)$ be a bimodule over it. Then $(M, S =\lambda ~\! \mathrm{id}_M + 2 Q)$ is a bimodule over the modified Rota-Baxter algebra $(A, R = \lambda ~ \! \mathrm{id}_A + 2P)$ of weight $\kappa = - \lambda^2$.
\end{proposition}

\begin{proof}
For any $a \in A$ and $u \in M$, we have
\begin{align*}
R(a) \cdot S(u ) =~& \big( \lambda a + 2P (a)  \big) \cdot \big(  \lambda u + 2Q (u) \big) \\
=~& \lambda^2 ~\! a \cdot u + 2 \lambda \big( P(a) \cdot u + a \cdot Q(u)  \big) + 4 ~\! P(a) \cdot Q(u) \\
=~& \lambda^2 ~\! a \cdot u + 2 \lambda \big( P(a) \cdot u + a \cdot Q(u)  \big) + 4 ~\! Q \big(  P(a) \cdot u + a \cdot Q(u) + \lambda ~\! a \cdot u \big) \\
=~& (\lambda ~\! \mathrm{id}_M + 2 Q) \big( (\lambda ~\! \mathrm{id}_A + 2 P)(a) \cdot u ~+~ a \cdot (\lambda \mathrm{id}_M + 2 Q)(u)  \big) ~-~ \lambda^2 ~\! a \cdot u \\
=~& S \big(   R(a) \cdot u + a \cdot S(u) \big) + \kappa ~\! a \cdot u. 
\end{align*}
Similarly, we can show that
\begin{align*}
S(u) \cdot R(a) = S \big( S(u) \cdot a + u \cdot R(a) \big) + \kappa ~\! u \cdot a.
\end{align*}
This completes the proof.
\end{proof}

\begin{proposition}
Let $(A, R)$ be a modified Rota-Baxter algebra of weight $\kappa$, and $\{ (M_i,S_i) \}_{i \in I}$be a family of bimodules over it. Then $(\oplus_{i \in I} M_i, \oplus_{i \in I} S_i)$ is also a bimodule over $(A, R)$.
\end{proposition}

\begin{proof}
The $A$-bimodule structure on $\oplus_{i \in I} M_i$ is given by the componentwise $A$-bimodule structure on each $M_i$. For any $a \in A$ and $(u_i)_{i \in I} \in \oplus_{i \in I} M_i$, we have
\begin{align*}
R(a) \cdot (\oplus_{i \in I} S_i)(u_i)_{i \in I} =~& R(a) \cdot (S_i (u_i))_{i \in I} \\
=~& \big( R(a) \cdot S_i (u_i) \big)_{i \in I} \\
=~& \big( S_i ( R(a) \cdot u_i + a \cdot S_i (u_i)) + \kappa ~ \! a \cdot u_i \big)_{i \in I} \\
=~& (\oplus_{i \in I} S_i) \bigg( R(a) \cdot (u_i)_{i \in I} + a \cdot \big( (\oplus_{i \in I} S_i) (u_i)_{i \in I} \big)  \bigg) + \kappa ~ \! a \cdot (u_i)_{i \in I}.
\end{align*}
This shows that $(\oplus_{i \in I} M_i, \oplus_{i \in I} S_i)$ is a bimodule over the modified Rota-Baxter algebra $(A, R)$ of weight $\kappa$.
\end{proof}

Recall that, if $M$ is an $A$-bimodule, then $\mathrm{End}(M)$ can be equipped with an $A$-bimodule structure with left and right $A$-actions
\begin{align*}
(a \odot f) (u) = f (u \cdot a) ~~~ \text{ and } ~~~ (f \odot a) (u) = f (a \cdot u), \text{ for } a \in A, f \in \mathrm{End}(M), u \in M.
\end{align*}

\begin{proposition}
Let $(A,R)$ be a modified Rota-Baxter algebra of weight $\kappa$, and $(M,S)$ be a bimodule over it. Then $(\mathrm{End}(M), \widetilde{S})$ is also a bimodule over $(A, R)$, where $\widetilde{S} : \mathrm{End}(M) \rightarrow \mathrm{End}(M)$ is given by $\widetilde{S} (f) = - f \circ S$, for $f \in \mathrm{End}(M)$.
\end{proposition}

\begin{proof}
For any $a \in A$, $f \in \mathrm{End}(M)$ and $u \in M$, we have
\begin{align*}
\big( R(a) \odot \widetilde{S}(f) \big)(u) = \widetilde{S}(f) \big(  u \cdot R(a) \big) 
=~& - f \big( S(u \cdot R(a))  \big) \\
=~& - f \big(  S(u) \cdot R(a) - S (S(u) \cdot a) - \kappa ~ \! u \cdot a  \big) \\
=~& - f \big(  S(u) \cdot R(a) \big) - \widetilde{S} (f) (S(u) \cdot a) + \kappa ~ \! f(u \cdot a) \\
=~& - \big( R(a) \odot f + a \odot \widetilde{S}(f)  \big) (S(u)) + \kappa ~ \! f (u \cdot a) \\
=~& \big( \widetilde{S} \big(  R(a) \odot f + a \odot \widetilde{S}(f)  \big) + \kappa ~ \! a \odot f   \big) (u)
\end{align*}
and
\begin{align*}
\big( \widetilde{S} (f) \odot R(a)  \big)(u) = \widetilde{S} (f) \big( R(a) \cdot u   \big) 
=~& - f \big(  S (R(a) \cdot u) \big) \\
=~& -f \big(  R(a) \cdot S(u) - S (a \cdot S(u)) - \kappa~ \! a \cdot u \big) \\
=~& -f \big(  R(a) \cdot S(u) \big) - \widetilde{S}(f) (a \cdot S(u)) + \kappa ~ \! f (a \cdot u) \\
=~& - \big( f \odot R(a) + \widetilde{S}(f) \odot a  \big)(S(u)) + \kappa ~ \! f (a \cdot u) \\
=~& \big(  \widetilde{S} \big( \widetilde{S}(f) \odot a + f \odot R(a) \big) + \kappa ~ \! f \odot a \big) (u).
\end{align*}
This shows that $(\mathrm{End}(M), \widetilde{S})$ is a bimodule over the modified Rota-Baxter algebra $(A, R)$ of weight $\kappa$.
\end{proof}

\begin{proposition}\label{semi-mrba}
(Semidirect product) Let $(A, R)$ be a modified Rota-Baxter algebra of weight $\kappa$, and $(M,S)$ be a bimodule over it. Then the pair $(A \oplus M, R \oplus S)$ is a modified Rota-Baxter algebra of weight $\kappa$, where the associative algebra structure on $A \oplus M$ is given by the semidirect product
\begin{align*}
(a, u) \cdot_\ltimes (b, v) = (a \cdot  b, a \cdot v + u \cdot b), \text{ for } (a, u), (b, v) \in A \oplus M. 
\end{align*}
\end{proposition}

\begin{proof}
For $(a,u), (b, v) \in A \oplus M$, we have
\begin{align*}
&(R \oplus S) (a,u) \cdot_\ltimes (R \oplus S)(b,v) \\
&= (R(a), S(u)) \cdot_\ltimes (R(b), S(v)) \\
&= \big(  R(a) \cdot R(b),~ R(a) \cdot S(v) + S(u) \cdot R(b) \big) \\
&= \bigg(  R ( R(a) \cdot  b + a \cdot  R(b)) + \kappa~ \! a \cdot  b, ~  S ( R(a) \cdot v + a \cdot S(v)) + \kappa ~ \! a \cdot v  + S ( S(u) \cdot b + u \cdot R(b)) + \kappa ~ \! u \cdot b \bigg) \\
&= (R \oplus S) \bigg(    \big( R(a) \cdot  b , ~ R(a) \cdot v + S(u) \cdot b   \big) +   \big(  a \cdot  R(b), ~ a \cdot S(v) + u \cdot R(b)  \big) \bigg)  + \kappa ~ \! \big( a \cdot  b, ~ a \cdot v + u \cdot b  \big)  \\
&= (R \oplus S) \bigg( (R \oplus S)(a,u) \cdot_\ltimes (b, v) ~+~ (a, u) \cdot_\ltimes (R \oplus S)(b,v)    \bigg) + \kappa ~ \! (a, u) \cdot_\ltimes (b, v).
\end{align*}
This shows that $R \oplus S$ is a modified Rota-Baxter operator of weight $\kappa$ on the semidirect product associative algebra $A \oplus M$. In other words, $(A \oplus M, R \oplus S)$ is a modified Rota-Baxter algebra of weight $\kappa$.
\end{proof}

In the following, we show that a modified Rota-Baxter algebra of weight $\kappa$ induces a new associative algebra structure. This associative algebra will be useful in the construction of the cohomology of the modified Rota-Baxter algebra.

\begin{proposition}
Let $(A, R)$ be a modified Rota-Baxter algebra of weight $\kappa$. Then we have the followings.

(i) The underlying vector space $A$ inherits a new associative algebra structure with the product
\begin{align*}
a \ast_R b = R(a) \cdot  b + a \cdot  R(b), ~ \text{ for } a, b \in A.
\end{align*}
(We denote this associative algebra simply by $A_R$).

(ii) The pair $(A_R, R)$ is a modified Rota-Baxter algebra of weight $\kappa$.
% and the map $R : A_R \rightarrow A$ is a morphism between modified Rota-Baxter algebras.
\end{proposition}

\begin{proof}
For any $a, b, c \in A$, we have
\begin{align*}
(a \ast_R b) \ast_R c =~& R(a \ast_R b) \cdot c + (a \ast_R b) \cdot R(c) \\
=~& \big( R(a) \cdot R(b) - \kappa ~ \! a \cdot b  \big) \cdot c + \big(  R(a) \cdot b + a \cdot R(b) \big) \cdot R(c) \\
=~& R(a) \cdot \big(  R(b) \cdot c + b \cdot R(c) \big) + a \cdot \big( R(b ) \cdot R(c) - \kappa ~ \! b \cdot c  \big) \\
=~& R(a) \cdot (b \ast_R c) + a \cdot R(b \ast_R c) \\
=~& a \ast_R (b \ast_R c).
\end{align*}
This shows that $A_R$ is an associative algebra. \\

(ii) For any $a, b \in A$, 
\begin{align*}
R(a) \ast_R R(b) =~& R^2 (a) \cdot R(b) + R(a) \cdot R^2 (b) \\
=~& R (R(a) \ast_R b) + \kappa ~ \! R(a) \cdot b + R (a \ast_R R(b)) + \kappa ~ \! a \cdot R(b) \\
=~& R \big( R(a) \ast_R b + a \ast_R R(b)  \big) + \kappa ~ \! a \ast_R b.
\end{align*}
This shows that $R$ is a modified Rota-Baxter operator of weight $\kappa$ on the associative algebra $A_R$. In other words, $(A_R, R)$ is a modified Rota-Baxter algebra of weight $\kappa$.
\end{proof}

\begin{proposition}\label{first-ar-bimod}
Let $(A, R)$ be a modified Rota-Baxter algebra of weight $\kappa$, and $(M,S)$ be a bimodule over it. 

(i) Then there is an $A_R$-bimodule structure on $M$ with left and right actions (denoted by the same notation) given by
\begin{align*}
a ~ {\cdot}_S ~ u := R(a) \cdot u + a \cdot S(u) \quad \text{ and } \quad u ~ {\cdot}_S ~ a := S(u) \cdot a + u \cdot R(a),
\end{align*}
for $a \in A_R$, $u \in M$. We denote this $A_R$-bimodule structure simply by ${M}_S$.

(ii) The pair $({M}_S, S)$ is bimodule over the modified Rota-Baxter algebra $(A_R,R)$ of weight $\kappa$.
\end{proposition}

\begin{proof}
(i) Consider the (semidirect product) modified Rota-Baxter algebra $(A \oplus M, R \oplus S)$ of weight $\kappa$ given in Proposition \ref{semi-mrba}. This induces an associative algebra structure on the direct sum $A \oplus B$ ~(which we denote by $(A \oplus B)_{R \oplus S}$) with the product given by
\begin{align*}
(a , u) \ltimes_{R \oplus S} (b, v) =~& (R(a), S(u))  \ltimes_{R \oplus S} (b, v) ~+~ (a, u)  \ltimes_{R \oplus S} (R(b), S(v)) \\
=~& \big( R(a) \cdot b + a \cdot R(b),~ R(a) \cdot v + S(u) \cdot b + a \cdot S(v) + u \cdot R(b)  \big) \\
=~& \big( a \ast_R b,~ a ~{\cdot}_S~ v + u ~{\cdot}_S ~ b  \big).
\end{align*}
This expression of the product says that ${M}_S$ is an $A_R$-bimodule.

(ii) For any $a \in A$ and $u \in M$, we have
\begin{align*}
R(a) ~{\cdot}_S~ S(u) =~& R^2(a) \cdot S(u) + R(a) \cdot S^2(u) \\
=~& S \big( R^2(a) \cdot u + R(a) \cdot S(u)  \big)+ \kappa ~ \! R(a) \cdot u + S \big( R(a) \cdot S(u) + a \cdot S^2(u)  \big) + \kappa ~ \! a \cdot S(u) \\
=~& S \big( R(a)~ {\cdot}_S ~ u \big) + \kappa ~ \! R(a) \cdot u + S \big(  a ~ {\cdot}_S ~S(u) \big) + \kappa ~ \! a \cdot S(u) \\
=~& S \big( R(a) ~{\cdot}_S ~ u + a ~ {\cdot}_S ~ S(u)  \big) + \kappa ~ \! a ~ {\cdot}_S ~ u
\end{align*}
and
\begin{align*}
S(u) ~ {\cdot}_S ~ R(a) =~& S^2(u) \cdot R(a) + S(u) \cdot R^2(a) \\
=~& S \big(  S^2(u) \cdot a + S(u) \cdot R(a) \big) + \kappa ~ \! S(u) \cdot a + S \big( S(u) \cdot R(a) + u \cdot R^2(a)  \big) + \kappa ~ \! u \cdot R(a) \\
=~& S \big(  S(u) ~ {\cdot}_S ~ a \big) + \kappa ~ \! S(u) \cdot a + S \big( u ~{\cdot}_S ~ R(a)  \big)  + \kappa ~ \! u \cdot R(a) \\
=~& S \big(  S(u) ~{\cdot}_S ~ a +  u ~{\cdot}_S~ R(a)   \big) + \kappa ~ \! u ~ {\cdot}_S ~ a.
\end{align*}
This shows that $({M}_S, S)$ is a bimodule over the modified Rota-Baxter algebra $(A_R, R)$.
\end{proof}

\begin{remark}
When $(M,S) = (A,R)$ is the adjoint bimodule over the modified Rota-Baxter algebra $(A,R)$ of weight $\kappa$, then the $A_R$-bimodule ${M}_S$ coincides with the adjoint $A_R$-bimodule. Hence we have $({M}_S, S) = (A_R, R)$ the adjoint bimodule over the modified Rota-Baxter algebra $(A_R, R)$ of weight $\kappa$.
\end{remark}

Given a bimodule over a modified Rota-Baxter algebra $(A,R)$, in Proposition \ref{first-ar-bimod}, we constructed a bimodule over the induced modified Rota-Baxter algebra $(A_R,R)$. In the following, we will construct another bimodule over $(A_R,R)$. This new construction will be helpful to construct the cohomology of the modified Rota-Baxter algebra.

\begin{proposition}
Let $(A,R)$ be a modified Rota-Baxter algebra of weight $\kappa$, and $(M,S)$ be a bimodule over it. Then $M$ can be equipped with an $A_R$-bimodule structure with left and right $A_R$-actions
\begin{align*}
 a ~ \widetilde{\cdot} ~u := R(a) \cdot u - S (a \cdot u) \quad \text{ and } \quad
u ~ \widetilde{\cdot} ~ a := u \cdot R(a) - S (u \cdot a),
\end{align*}
for $a \in A_R$, $u \in M$. (We denote this $A_R$-bimodule structure simply by $\widetilde{M}$). Moreover, the pair $(\widetilde{M}, S)$ is a bimodule over the modified Rota-Baxter algebra $(A_R,R)$ of weight $\kappa$.
\end{proposition}

\begin{proof}
For $a, b \in A_R$ and $u \in M$, we have
\begin{align*}
(a \ast_R b) ~ \widetilde{\cdot}~ u 
=~& R (a \ast_R b ) \cdot u - S \big(  (a \ast_R b) \cdot u \big) \\
=~& \big( R(a) \cdot R(b) - \kappa ~ \! a \cdot b  \big) \cdot u - S \big( (R(a) \cdot b + a \cdot R(b)) \cdot u  \big) \\
=~& R(a) \cdot R(b) \cdot u - \kappa ~ \! a \cdot b \cdot u - S \big(  R(a) \cdot b \cdot u + a \cdot S(b \cdot u) \big) - S \big( a \cdot R(b) \cdot u - a \cdot S(b \cdot u)  \big) \\
=~& R(a) \cdot (R(b) \cdot u) - R(a) \cdot S(b \cdot u) - S \big( a \cdot R(b) \cdot u - a \cdot S(b \cdot u)  \big) \\
=~& R(a) \cdot \big( R (b) \cdot u - S(b \cdot u) \big) - S \big(    a \cdot (R(b) \cdot u - S(b \cdot u)) \big) \\
=~& a ~\widetilde{\cdot} ~ \big( R(b) \cdot u - S(b \cdot u) \big) \\
=~& a ~\widetilde{\cdot} ~ (b ~ \widetilde{\cdot} ~ u).
\end{align*}
Similarly,
\begin{align*}
(a ~ \widetilde{\cdot} ~ u) ~ \widetilde{\cdot} ~ b =~& (a ~ \widetilde{\cdot} ~ u) \cdot R(b) - S \big( (a ~ \widetilde{\cdot} ~ u) \cdot b  \big) \\
=~& \big(  R(a) \cdot u - S (a \cdot u) \big) \cdot R(b) - S \big( ( R(a) \cdot u - S(a \cdot u) ) \cdot b  \big) \\
=~& R(a) \cdot u \cdot R(b) - S \big( S(a\cdot u) \cdot b + a \cdot u \cdot R(b)  \big) - \kappa ~ \! a \cdot u \cdot b - S \big( R(a) \cdot u \cdot b  \big) + S\big(  S(a \cdot u) \cdot b \big)\\
=~& R(a) \cdot u \cdot R(b) - S \big( R(a) \cdot u \cdot b + a \cdot S(u \cdot b)  \big) - \kappa ~ \! a \cdot u \cdot b - S \big( a \cdot u \cdot R(b)  \big) + S \big( a \cdot S(u \cdot b)  \big) \\
=~& R(a) \cdot \big( u \cdot R(b) - S (u \cdot b)  \big) - S \big( a \cdot (u \cdot R(b) - S(u \cdot b))  \big) \\
=~& R(a) \cdot (u ~ \widetilde{\cdot}~ b) - S \big( a \cdot (u ~ \widetilde{\cdot} ~ b) \big) \\
=~& a ~ \widetilde{\cdot} ~ (u ~ \widetilde{\cdot} ~ b)
\end{align*}
and
\begin{align*}
(u ~ \widetilde{\cdot} ~ a ) ~ \widetilde{\cdot} ~ b =~& (u ~ \widetilde{\cdot} ~ a ) \cdot R(b) - S \big( (u ~ \widetilde{\cdot} ~ a) \cdot b \big) \\
=~& \big( u \cdot R(a) - S(u \cdot a) \big) \cdot R(b) - S \big( (u \cdot R(a) - S(u \cdot a)) \cdot b  \big) \\
=~& u \cdot R(a) \cdot R(b) - S(u \cdot a) \cdot R(b) - S \big( u \cdot R(a) \cdot b  \big) + S \big( S(u \cdot a) \cdot b  \big) \\
=~& u \cdot R(a) \cdot R(b) - S \big( S(u \cdot a) \cdot b + u \cdot a \cdot R(b)  \big) - \kappa ~ \! u \cdot a \cdot b - S \big( u \cdot R(a) \cdot b  \big) + S \big( S(u \cdot a) \cdot b  \big) \\
=~& u \cdot \big( R(a) \cdot R(b) - \kappa~ \! a \cdot b  \big) - S \big(  u \cdot R(a) \cdot b + u \cdot a \cdot R(b) \big) \\
=~& u \cdot R(a \ast_R b) - S \big( u \cdot (a \ast_R b)  \big) \\
=~& u ~ \widetilde{\cdot} ~ (a \ast_R b).
\end{align*}
This shows that $\widetilde{M}$ is an $A_R$-bimodule. For the second part, we observe that
\begin{align*}
R(a) ~ \widetilde{\cdot} ~ S(u) =~& R^2(a) \cdot S(u) - S \big(  R(a) \cdot S(u) \big) \\
=~& S \big( R^2(a) \cdot u + R(a) \cdot S(u) \big) + \kappa ~ \! R(a) \cdot u - S^2 \big( R(a) \cdot u + a \cdot S(u)  \big) - \kappa ~ \! S (a \cdot u) \\
=~& S \big( R^2(a) \cdot u - S (R(a) \cdot u) + R(a) \cdot S(u) - S (a \cdot S(u))  \big) + \kappa ~ \! \big( R(a) \cdot u - S (a \cdot u)  \big) \\
=~& S \big(  R(a) ~ \widetilde{\cdot} ~ u + a ~ \widetilde{\cdot} ~ S(u) \big) + \kappa ~ \! a ~ \widetilde{\cdot} ~ u
\end{align*}
and
\begin{align*}
S(u) ~ \widetilde{\cdot} ~ R(a) =~& S(u) \cdot R^2(a) - S \big( S(u) \cdot R(a)  \big) \\
=~& S \big( S(u) \cdot R(a) + u \cdot R^2 (a)  \big) + \kappa ~ \! u \cdot R(a) -S^2 \big(  S(u) \cdot a + u \cdot R(a) \big) - \kappa ~ \! S (u \cdot a) \\
=~& S \big( S(u) \cdot R(a) - S (S(u) \cdot a) + u \cdot R^2(a) - S (u \cdot R(a))  \big) + \kappa ~ \! \big( u \cdot R(a) - S (u \cdot a)  \big) \\
=~& S \big( S(u) ~ \widetilde{\cdot} ~ a + u ~ \widetilde{\cdot} ~ R(a)  \big) + \kappa ~ \! u ~ \widetilde{\cdot} ~ a,
\end{align*}
which shows that $(\widetilde{M}, S)$ is a bimodule over the modified Rota-Baxter algebra $(A_R, R)$ of weight $\kappa$. 
\end{proof}

\section{Cohomology of modified Rota-Baxter algebras}\label{sec-3}
In this section, we introduce the cohomology of a modified Rota-Baxter algebra $(A, R)$ with coefficients in a bimodule $(M,S)$. This cohomology is obtained as a biproduct of the Hochschild cohomology of the underlying associative algebra $A$ and the Hochschild cohomology of the induced associative algebra $A_R$. We also compare our cohomology of a modified Rota-Baxter algebra with the cohomology of a Rota-Baxter algebra defined in \cite{DasSK,wang-zhou}.

Let $(A, R)$ be a modified Rota-Baxter algebra of weight $\kappa$, and $(M,S)$ be a bimodule over it. First, consider the Hochschild cochain complex of $A$ with coefficients in the $A$-bimodule $M$. The complex is given by $\{ C^\bullet (A, M), \delta_\mathrm{Hoch} \}$, where $C^n (A, M) = \mathrm{Hom} (A^{\otimes n}, M)$ for $n \geq 0$, and the coboundary map $\delta_\mathrm{Hoch} : C^n (A, M) \rightarrow C^{n+1} (A, M)$ is given by
\begin{align*}
(\delta_\mathrm{Hoch} f) (a_1, \ldots, a_{n+1}) =~& (-1)^{n+1} ~ a_1 \cdot f (a_2, \ldots, a_{n+1}) + f(a_1, \ldots, a_n) \cdot a_{n+1} \\
~&+ \sum_{i=1}^n (-1)^{i+n+1} ~ f (a_1, \ldots, a_i \cdot a_{i+1}, \ldots , a_{n+1}),
\end{align*}
for $f \in C^n (A, M)$ and $a_1, \ldots, a_{n+1} \in A$. We denote the corresponding cohomology groups by $H^\bullet (A, M)$. Second,  we consider the Hochschild cochain complex $\{ C^\bullet (A_R, \widetilde{M}), \widetilde{\delta}_\mathrm{Hoch} \}$ of the induced associative algebra $A_R$ with coefficients in the $A_R$-bimodule $\widetilde{M}$, where $C^n (A_R, \widetilde{M}) = \mathrm{Hom} (A^{\otimes n}, M)$ for $n \geq 0$, and the coboundary map $\widetilde{\delta}_\mathrm{Hoch} : C^n (A_R, \widetilde{M}) \rightarrow C^{n+1} (A_R, \widetilde{M})$ is given by
%the same cochain groups with the coboundary map $\partial_{\mathrm{Hoch}} : C^n (A, M) \rightarrow C^{n+1} (A, M)$ for $n \geq 0$, given by
\begin{align*}
(\widetilde{\delta}_\mathrm{Hoch} f)(a_1, \ldots, a_{n+1} ) =~& (-1)^{n+1} R(a_1) \cdot f(a_2, \ldots, a_{n+1}) - (-1)^{n+1} ~S \big( a_1 \cdot f(a_2, \ldots, a_{n+1}) \big) \\
&+ f (a_1, \ldots, a_n) \cdot R(a_{n+1}) - S \big( f(a_1, \ldots, a_n) \cdot a_{n+1} \big) \\
&+ \sum_{i=1}^n (-1)^{i+n+1}~ f \big( a_1, \ldots, R(a_i) \cdot a_{i+1} + a_i \cdot R(a_{i+1}), \ldots, a_{n+1}  \big),
\end{align*}
for $f \in C^n (A_R, \widetilde{M})$ and $a_1, \ldots, a_{n+1} \in A$. The cohomology groups of this cochain complex are denoted by $H^\bullet_\mathrm{Hoch} (A_R, \widetilde{M})$.

\begin{remark}
Let $(M, S) = (A, R)$ be the adjoint bimodule over the modified Rota-Baxter algebra $(A, R)$ of weight $\kappa$. Then the cochain complex constructed above is given by $\{ C^\bullet (A_R, \widetilde{A}), \widetilde{\delta}_\mathrm{Hoch} \}$, which is the Hochschild cochain complex of the induced algebra $A_R$ with coefficients in the $A_R$-bimodule $\widetilde{A}$. It can be easily check that the corresponding cohomology governs the formal deformation theory of the modified Rota-Baxter operator $R$ keeping the underlying associative structure on $A$ intact. In the following, we will be interested in a more general cohomology that simultaneously controls the formal deformations of the underlying algebra and the modified Rota-Baxter operator.
\end{remark}

Let $(A,R)$ be a modified Rota-Baxter algebra of weight $\kappa$, and $(M,S)$ be a bimodule over it. Our construction of the general cohomology is based on the following result. This is highly motivated by the similar result in the context of Rota-Baxter algebras \cite[Proposition 5.2]{wang-zhou}. The proof can be done by considering the correspondence between Rota-Baxter algebras and modified Rota-Baxter algebras, and using the above mentioned reference.

\begin{proposition}
Let $(A, R)$ be a modified Rota-Baxter algebra of weight $\kappa$, and $(M,S)$ be a bimodule.  Then the collection of maps $\{ \Psi^n : C^n (A, M) \rightarrow C^n (A_R, \widetilde{M}) \}_{n \geq 0}$ given by
\begin{align*}
\Psi^0 =~& \mathrm{id}_M, ~ \text{ and } \\
\Psi^n (f) (a_1, \ldots, a_n ) =~& f \big( R(a_1), \ldots, R(a_n) \big) \\
~&- \sum_{1 \leq i_1 < \cdots < i_r \leq n , , r \text{ odd }} (- \kappa)^{\frac{r-1}{2}} S \circ f \big( R(a_1), \ldots, a_{i_1}, \ldots, a_{i_2}, \ldots, a_{i_r}, \ldots, R(a_n)  \big) \\
~& - \sum_{1 \leq i_1 < \cdots < i_r \leq n , , r \text{ even }} (- \kappa)^{\frac{r}{2} + 1} S \circ f \big( R(a_1), \ldots, a_{i_1}, \ldots, a_{i_2}, \ldots, a_{i_r}, \ldots, R(a_n)  \big)
\end{align*}
defines a morphism of cochain complexes from $\{ C^\bullet (A, M), \delta_\mathrm{Hoch} \}$ to $\{ C^\bullet (A_R, \widetilde{M}), \widetilde{\delta}_\mathrm{Hoch} \}$.
\end{proposition}

The above proposition suggests to consider a new cochain complex $\{ C^\bullet_\mathrm{mRBA} ((A,R), (M,S)), \delta_\mathrm{mRBA} \}$, where
\begin{align*}
C^n_\mathrm{mRBA} ((A,R), (M,S)) = \begin{cases} M & \text{ if } n =0,\\
C^n (A, M) \oplus C^{n-1}(A_R,\widetilde{M}) = \mathrm{Hom} (A^{\otimes n}, M) \oplus \mathrm{Hom}(A^{\otimes n-1}, {M}) & \text{ if } n \geq 1.
\end{cases}
\end{align*}
The coboundary operator $\delta_\mathrm{mRBA} : C^n_\mathrm{mRBA} ((A,R), (M,S)) \rightarrow C^{n+1}_\mathrm{mRBA} ((A,R), (M,S))$ is given by
\begin{align*}
\delta_\mathrm{mRBA} (u) =~& (\delta_\mathrm{Hoch}(u), -u), ~\text{ for } u \in M,\\
\delta_\mathrm{mRBA} (\chi, \Phi) =~& \big( \delta_\mathrm{Hoch} (\chi), ~ - \widetilde{\delta}_\mathrm{Hoch} (\Phi) - \Psi^n (\chi)  \big), ~\text{ for } (\chi, \Phi) \in C^n_\mathrm{mRBA} ((A, R), (M,S)).
\end{align*}
Note that $(\delta_\mathrm{mRBA})^2 = 0$ as we have
\begin{align*}
(\delta_\mathrm{mRBA})^2 (u) =~& \delta_\mathrm{mRBA} ( \delta_\mathrm{Hoch}(u), -u ) \\
=~& \big( (\delta_\mathrm{Hoch})^2(u), ~ \widetilde{\delta}_\mathrm{Hoch} (u) - \Psi^1 \circ \delta_\mathrm{Hoch}(u)  \big) = 0
\end{align*}
and 
\begin{align*}
(\delta_\mathrm{mRBA})^2 (\chi, \Phi) =~& \delta_\mathrm{mRBA} \big( \delta_\mathrm{Hoch} (\chi), ~ - \widetilde{\delta}_\mathrm{Hoch} (\Phi) - \Psi^n (\chi)   \big) \\
=~& \big( (\delta_\mathrm{Hoch})^2 (\chi), ~ (\widetilde{\delta}_\mathrm{Hoch})^2 (\Phi) + \widetilde{\delta}_\mathrm{Hoch} \circ \Psi^n (\chi) - \Psi^{n+1} \circ \delta_\mathrm{Hoch} (\chi)  \big) = 0.
\end{align*}

It is clear from the above description that a pair $(\chi, \Phi) \in \mathrm{Hom}(A^{\otimes 2}, M) \oplus \mathrm{Hom}(A, M)$ is a $2$-cocycle if they satisfy
\begin{align*}
\begin{cases}
a \cdot \chi (b, c) - \chi (a \cdot b, c) + \chi (a, b \cdot c) - \chi (a, b) \cdot c = 0, \\\\
R(a) \cdot \Phi (b) - S (a \cdot \Phi (b)) - \Phi \big( R(a) \cdot b + a \cdot R(b) \big) + \Phi (a) \cdot R (b) - S (\Phi (a) \cdot b) \\
 \quad + \chi (R(a) , R(b)) - \kappa ~ \! \chi (a, b) - R \big( \chi(R(a), b) + \chi (a, R(b))  \big) = 0.
\end{cases}
\end{align*}
Further, it is a $2$-coboundary if there exists a pair $(\varphi_1, u) \in \mathrm{Hom}(A, M) \oplus M$ such that 
\begin{align*}
\chi = \delta_\mathrm{Hoch} (\varphi_1) ~~~ \text{ and } ~~~ \Phi = - \widetilde{\delta}_\mathrm{Hoch}(u) - \Psi^1 (\varphi_1).
\end{align*}

\medskip

Let $Z^n_\mathrm{mRBA} ((A,R), (M,S))$ denote the space of $n$-cocycles, and $B^n_\mathrm{mRBA} ((A,R),(M,S))$ denote the space of $n$-coboundaries. Then we have $B^n_\mathrm{mRBA} ((A,R),(M,S)) \subset Z^n_\mathrm{mRBA} ((A,R), (M,S))$, for $n \geq 0$. The quotient groups
\begin{align*}
H^n_\mathrm{mRBA} ((A,R), (M,S)) := \frac{Z^n_\mathrm{mRBA}  ((A,R), (M,S))}{B^n_\mathrm{mRBA} ((A,R), (M,S))}, \text{ for } n \geq 0
\end{align*}
are called the cohomology of the modified Rota-Baxter algebra $(A, R)$ of weight $\kappa$ with coefficients in the bimodule $(M,S)$.

\medskip

It is easy to see that there is a short exact sequence
\begin{align*}
0 \rightarrow \big\{ C^{\bullet -1} (A_R, \widetilde{M}), \widetilde{\delta}_\mathrm{Hoch} \big\} \xrightarrow{\iota} \big\{ C^\bullet_\mathrm{mRBA} ((A, R), (M,S)), \delta_\mathrm{mRBA}  \big\} \xrightarrow{p} \big\{ C^\bullet (A, M) , \delta_\mathrm{Hoch}  \big\} \rightarrow 0
\end{align*}
of cochain complexes, where the maps $\iota$ and $p$ are given by $\iota (\Phi) = (0, (-1)^{n-1} \Phi)$, for $\Phi \in C^{n-1} (A_R, \widetilde{M})$ and $p (\chi, \Phi) = \chi$, for $(\chi, \Phi) \in C^n_\mathrm{mRBA} ((A,R), (M,S))$. As a consequence, we obtain the following.

\begin{theorem}
Let $(A, R)$ be a modified Rota-Baxter algebra of weight $\kappa$, and $(M,S)$ be a bimodule over it. Then there is a long exact sequence connecting various cohomology groups
\begin{align*}
0 \rightarrow H^0_\mathrm{mRBA} ((A, R), (M,S)) \rightarrow H^0_\mathrm{Hoch} (A,M) \rightarrow H^0_\mathrm{Hoch} (A_R, \widetilde{M}) \rightarrow H^1_\mathrm{mRBA}((A,R), (M,S)) \rightarrow \cdots.
\end{align*}
\end{theorem}

\medskip

\noindent {\bf A relation with the cohomology of Rota-Baxter algebras.} In \cite{DasSK,wang-zhou} the authors introduced the cohomology of Rota-Baxter algebras. Here we will show that the cohomology of a Rota-Baxter algebra $(A, P)$ of weight $\lambda$ is isomorphic to the cohomology of the modified Rota-Baxter algebra $(A, R = \lambda ~ \! \mathrm{id} + 2P)$ of weight $\kappa = -\lambda^2$.

We first recall the cohomology of Rota-Baxter algebras from \cite{wang-zhou}. Let $(A, P)$ be a Rota-Baxter algebra of weight $\lambda$. Then the vector space $A$ carries an associative algebra structure with the multiplication given by 
\begin{align*}
a \ast_P b := P(a) \cdot b ~+~ a \cdot P(b) + \lambda ~ \! a \cdot b, \text{ for } a, b \in A.
\end{align*}
We denote this associative algebra simply by $A_P$.
Moreover, if $(M, Q)$ is a bimodule over the Rota-Baxter algebra $(A, P)$ of weight $\lambda$, then the maps
\begin{align*}
A \otimes M \rightarrow M, ~(a, u) \mapsto P(a) \cdot u - Q (a \cdot u), \\
M \otimes A \rightarrow M, ~ (u, a) \mapsto u \cdot P(a) - Q (u \cdot a)
\end{align*}
defines an $A_P$-bimodule structure on $M$ (denoted by $\overline{M}$). Then the authors considered the cochain complex $\{ C^\bullet_\mathrm{RBA} ((A, P), (M, Q)), \delta_\mathrm{RBA} \}$, where
\begin{align*}
C^n_\mathrm{RBA} ((A, P), (M, Q)) = \begin{cases}
M & \text{ if } n \geq 0,\\
C^n (A, M) \oplus C^{n-1} (A_P, \overline{M}) = \mathrm{Hom}(A^{\otimes n}, M) \oplus \mathrm{Hom}(A^{\otimes n-1}, M) & \text{ if } n \geq 1.
\end{cases}
\end{align*}
The coboundary map is given by
\begin{align*}
\delta_\mathrm{RBA} (u) =~& \big( \delta_\mathrm{Hoch} (u), -u  \big), \\
\delta_\mathrm{RBA} (f, g) =~& \big( \delta_\mathrm{Hoch} (f), ~ - \overline{\delta}_\mathrm{Hoch} (g) - \Phi^n (f) \big),
\end{align*}
where $\overline{\delta}_\mathrm{Hoch}$ is the Hochschild coboundary operator of the associative algebra $A_P$ with coefficients in the $A_P$-bimodule $\overline{M}$, and $\{ \Phi^n : C^n(A, M) \rightarrow C^n (A_P, \overline{M}) \}_{n \geq 0}$ is the sequence of maps given by $\Phi^0 = \mathrm{id}_M$ and
\begin{align*}
\Phi^n (f) (a_1, \ldots, a_n ) = f \big( P(a_1), \ldots, P(a_n)  \big) - \sum_{k=0}^{n-1} \lambda^{n-k-1} \sum_{1 \leq i_1 < \cdots < i_k \leq n} ~ Q \circ f \big(  a_1, \ldots, P (a_{i_1}), \ldots, P(a_{i_k}), \ldots, a_n \big).
\end{align*}
The cohomology groups of the cochain complex $\{ C^\bullet_\mathrm{RBA} ((A, P), (M, Q)), \delta_\mathrm{RBA} \}$ are called the cohomology of the Rota-Baxter algebra $(A, P)$ with coefficients in the bimodule $(M, Q).$

\begin{theorem}
Let $(A, P)$ be a Rota-Baxter algebra of weight $\lambda$, and $(M, Q)$ be a bimodule over it. Then the cohomology of the Rota-Baxter algebra $(A, P)$ with coefficients in $(M, Q)$ is isomorphic to the cohomology of the modified Rota-Baxter algebra $(A, R = \lambda ~ \! \mathrm{id}_A + 2 P)$ with coefficients in the bimodule $(M, S = \lambda ~ \! \mathrm{id}_M + 2 Q)$.
\end{theorem}

\begin{proof}
For each $n \geq 0$, we define an isomorphism of vector spaces
\begin{align*}
\Theta_n : C^n_\mathrm{RBA} ((A,P), (M, Q)) \rightarrow C^n_\mathrm{mRBA} ((A, R), (M, S))
\end{align*}
by
\begin{align*}
\Theta_0 (u) =~& \frac{u}{2}, ~\text{ for } u \in C^0_\mathrm{RBA} ((A,P), (M,Q)) = M, \\
\Theta_n (f, g) =~& (f,~ 2^{n-2} g), ~ \text{ for } (f, g) \in C^{n \geq 1}_\mathrm{RBA} ((A, P), (M, Q)).
\end{align*}
Then it is straightforward to verify that $\delta_\mathrm{mRBA} \circ \Theta_n = \Theta_{n+1} \circ \delta_\mathrm{RBA}$, for all $n \geq 0$. In other words, the collections $\{ \Theta_n \}_{n \geq 0}$ defines an isomorphism of cochain complexes from $\{ C^\bullet_\mathrm{RBA} ((A,P), (M,Q)), \delta_\mathrm{RBA} \}$ to $\{ C^\bullet_\mathrm{RBA} ((A,R), (M,S)), \delta_\mathrm{mRBA} \}$. Hence the result follows.
\end{proof}

\section{Deformations of modified Rota-Baxter algebras}\label{sec-4}
In this section, we study formal one-parameter deformations of modified Rota-Baxter algebras. In particular, we show that the infinitesimal in a formal deformation of a modified Rota-Baxter algebra $(A, R)$ is a $2$-cocycle in the cohomology complex of $(A,R)$ with coefficients in itself. Further, the vanishing of the second cohomology implies the rigidity of the modified Rota-Baxter algebra.

Let $(A, R)$ be a modified Rota-Baxter algebra of weight $\kappa$. Let $\mu$ denotes the associative multiplication on $A$, i.e. $\mu (a, b) = a \cdot b$, for $a, b \in A$. Consider the space $A[[t]]$ of formal power series in $t$ with coefficients from $A$. Then $A[[t]]$ is a $\mathbf{k}[[t]]$-module.

\begin{definition}
A {\bf formal one-parameter deformation} of the modified Rota-Baxter algebra $(A, R)$ of weight $\kappa$ consists of two formal power series of the form
\begin{align*}
&\mu_t = \sum_{i=0}^\infty \mu_i t^i, \text{ where } \mu_i \in \mathrm{Hom} (A^{\otimes 2}, A) \text{ with } \mu_0 = \mu,\\
&R_t = \sum_{i=0}^\infty R_i t^i, \text{ where } R_i \in \mathrm{Hom} (A, A) \text{ with } R_0 = R,
\end{align*}
such that the $\mathbf{k}[[t]]$-module $A[[t]]$ is an associative algebra with the $\mathbf{k}[[t]]$-bilinear multiplication $\mu_t$ and the $\mathbf{k}[[t]]$-linear map $R_t : A[[t]] \rightarrow A[[t]]$ is a modified Rota-Baxter operator of weight $\kappa$. In other words, $\big( A[[t]] = (A[[t]], \mu_t), R_t \big)$ is a modified Rota-Baxter algebra of weight $\kappa$. 

We often denote a formal one-parameter deformation as above by the pair $(\mu_t, R_t)$. 
\end{definition}

It follows that $(\mu_t, R_t)$ is a formal one-parameter deformation of the modified Rota-Baxter algebra $(A, R)$ of weight $\kappa$ if and only if the followings are hold:
\begin{align*}
\mu_t \big( \mu_t (a, b), c \big) =~& \mu_t \big( a, \mu_t(b,c)  \big), \\
\mu_t \big(  R_t (a), R_t (b) \big) =~& R_t \big(  \mu_t (R_t (a), b) + \mu_t (a, R_t(b)) \big) + \kappa ~ \mu_t (a, b),
\end{align*}
for all $a,b,c \in A$. By expanding these equations and comparing the coefficients of $t^n$ (for $n \geq 0$) in both sides, we obtain 
\begin{align*}
\sum_{i+j=n} \mu_i \big( \mu_j (a, b), c  \big) =~& \sum_{i+j = n} \mu_i \big(  a, \mu_j (b,c) \big), \\
\sum_{i+j+k = n} \mu_i \big( R_j (a), R_j (b)   \big) =~& \sum_{i+j+k = n} R_i \big( \mu_j (R_k (a), b) ~+~ \mu_j (a, R_k (b))  \big) + \kappa ~ \mu_n (a, b),
\end{align*}
for $n \geq 0$. The equations are obviously hold for $n = 0$ as $\mu$ is the associative multiplication on $A$ and the linear map $R: A \rightarrow A$ is a modified Rota-Baxter algebra of weight $\kappa$. However, for $n = 1$, we get
\begin{align}
\mu_1 (a \cdot b, c) + \mu_1 (a,b) \cdot c =~& \mu_1 (a, b \cdot c) + a \cdot \mu_1 (b, c), \label{inf-1-2}\\
R_1 (a) \cdot R(b) + R(a) \cdot R_1(b) + \mu_1 (R(a) , R(b)) =~& R_1 \big( R(a) \cdot b + a \cdot R(b)  \big) + R \big( \mu_1 (R(a), b) ~+~ \mu_1 (a, R(b))  \big) \label{inf-2-2} \\
~&+ R \big(  R_1(a) \cdot b + a \cdot R_1(b) \big) + \kappa ~ \mu_n (a, b), \nonumber
\end{align}
for all $a, b, c \in A$.
Note the the equation (\ref{inf-1-2}) is equivalent to $(\delta_\mathrm{Hoch} \mu_1) (a, b, c) = 0$ while the equation (\ref{inf-2-2}) is equivalent to 
\begin{align*}
\widetilde{\delta}_\mathrm{Hoch} (R_1) + \Psi^2 (\mu_1) = 0.
\end{align*}
Thus, it follows that
\begin{align}
\delta_\mathrm{mRBA} (\mu_1, R_1) = ( \delta_\mathrm{Hoch} (\mu_1), ~ - \widetilde{\delta}_\mathrm{Hoch} (R_1) - \Psi^2 (\mu_1)) = 0.
\end{align}
This shows that $(\mu_1, R_1)$ is a $2$-cocycle in the cohomology complex of $(A, R)$ with coefficients in itself. This is called the {\bf infinitesimal} of the formal one-parameter deformation $(\mu_t, R_t)$.

\begin{definition}
Let $(\mu_t, R_t)$ and $(\mu_t', R_t')$ be two formal one-parameter deformations of the modified Rota-Baxter algebra $(A, R)$ of weight $\kappa$. These two deformations are said to be {\bf equivalent} if there exists a formal isomorphism
\begin{align*}
\varphi_t = \sum_{i=0}^\infty \varphi_i t^i : \big( A[[t]] = (A[[t]], \mu_t), R_t  \big)  \rightarrow  \big( A[[t]]' = (A[[t]], \mu_t'), R_t'  \big) ~~ \text{ with } \varphi_0 = \mathrm{id}_A
\end{align*}
between modified Rota-Baxter algebras of weight $\kappa$, i.e.
\begin{align}\label{equiv-dend}
\varphi_t \big(   \mu_t (a, b) \big) = \mu_t' \big(  \varphi_t (a), \varphi_t(b) \big)  ~~~~ \text{ and } ~~~~ \varphi_t \circ R_t = R_t' \circ \varphi_t, ~ \text{ for } a, b \in A.
\end{align}
\end{definition}

By expanding the equations in (\ref{equiv-dend}) and comparing the coefficients of $t^n$ (for $n \geq 0$) in both sides, we obtain
\begin{align*}
\sum_{i+j=n} \varphi_i \big(  \mu_j (a, b) \big) =~& \sum_{i+j +k = n} \mu_i' \big(  \varphi_j (a), \varphi_k (b) \big),\\
\sum_{i+j = n} \varphi_i \circ R_j =~& \sum_{i+j = n} R_i' \circ \varphi_j,
\end{align*}
for $n \geq 0$. The equations are hold for $n=0$ as $\varphi_0 = \mathrm{id}_A$. However, for $n=1$, we get
\begin{align}
\mu_1 (a, b) + \varphi_1 (a \cdot b) =~& \mu_1' (a, b) + \varphi_1 (a) \cdot b + a \cdot \varphi_1 (b),  \label{mor-equiv-1}\\
R_1 + \varphi_1 \circ R =~& R_1' +  R \circ \varphi_1,  \label{mor-equiv-2}.
\end{align}
Note that the equation (\ref{mor-equiv-1}) is same as $\mu_1 - \mu_1' = \delta_\mathrm{Hoch} (\varphi_1)$, while the equation (\ref{mor-equiv-2}) is same as
\begin{align*}
R_1 - R_1' = - \Psi^1 (\varphi_1).
\end{align*}
Thus, we have
\begin{align*}
(\mu_1,R_1) - (\mu_1' , R_1') = \big( \delta_\mathrm{Hoch} (\varphi_1), ~ - \Psi^1 (\varphi_1)  \big) = \delta_\mathrm{mRBA} (\varphi_1, 0).
\end{align*}
As a consequence of the above discussions, we obtain the following.

\begin{theorem}
Let $(A, R)$ be a modified Rota-Baxter algebra of weight $\kappa$. Suppose $(\mu_t, R_t)$ is a formal one-parameter deformation of $(A, R).$ Then the infinitesimal is a $2$-cocycle in the cohomology complex of $(A, R)$ with coefficients in itself. Moreover, the corresponding cohomology class depends only on the equivalence class of the deformation.
\end{theorem}

We end this section by considering the rigidity of a modified Rota-Baxter algebra of weight $\kappa$. We also find a sufficient condition for rigidity.

\begin{definition}
A modified Rota-Baxter algebra $(A, R)$ is said to be {\bf rigid} if any formal one-parameter deformation $(\mu_t, R_t)$ is equivalent to the undeformed one $(\mu, R)$.
\end{definition}

\begin{theorem}
Let $(A, R)$ be a modified Rota-Baxter algebra of weight $\kappa$. If $H^2_\mathrm{mRBA} ((A,R), (A,R)) = 0$ then $(A,R)$ is rigid.
\end{theorem}

\begin{proof}
Let $(\mu_t, R_t)$ be any formal one-parameter deformation of the modified Rota-Baxter algebra $(A, R)$ of weight $\kappa$. Then the infinitesimal $(\mu_1, R_1)$ is a $2$-cocycle. Thus, it follows from the hypothesis that
\begin{align}
(\mu_1, R_1) = \delta_\mathrm{mRBA} (\varphi_1, a), \text{ for some } (\varphi_1, a) \in C^1_\mathrm{mRBA}((A, R), (A, R)).
\end{align}
Note that 
\begin{align*}
\delta_\mathrm{mRBA} (\varphi_1 + \delta_\mathrm{Hoch} (a), 0) =~& \big(  \delta_\mathrm{Hoch}(\varphi_1), ~- \Psi^1 (\varphi_1) - \Psi^1 \circ \delta_\mathrm{Hoch}(a) \big) \\
=~& \big( \delta_\mathrm{Hoch}(\varphi_1), ~ - \Psi^1 (\varphi_1) - \widetilde{\delta}_\mathrm{Hoch} (a)  \big) = \delta_\mathrm{mRBA} (\varphi_1, a).
\end{align*}
Thus, without loss of any generality, we may assume that $(\mu_1, R_1) = \delta_\mathrm{mRBA} (\varphi_1, 0)$, for some $(\varphi_1, 0) \in C^1_\mathrm{mRBA} ((A,R), (A, R))$. Then we define a map $\varphi_t = \mathrm{id}_A + t \varphi_1 : A [[t]] \rightarrow A[[t]]$ and consider the pair
\begin{align*}
\big( \overline{\mu}_t = \varphi_t \circ \mu_t \circ (\varphi_t^{-1} \otimes \varphi_t^{-1}), ~ \overline{R}_t = \varphi_t \circ R_t \circ \varphi_t^{-1} \big).
\end{align*}
Then $( \overline{\mu}_t, \overline{R}_t)$ is a formal one-parameter deformation equivalent to $(\mu_t, R_t)$. Moreover, it follows that $\overline{\mu}_t$ and $\overline{R}_t$ are of the following forms
\begin{align*}
\overline{\mu}_t = \mu + \overline{\mu}_2 t^2 + \cdots ~~~~ \text{ and } ~~~~ \overline{R}_t = R + \overline{R}_2 t^2 + \cdots .
\end{align*}
By repeating this argument, one can conclude that $(\mu_t, R_t)$ is equivalent to $(\mu, R)$. This completes the proof.
\end{proof}

\section{Abelian extensions of modified Rota-Baxter algebras}\label{sec-5}
In this section, we define and study abelian extensions of a modified Rota-Baxter algebra $(A,R)$ by a bimodule $(M,S)$. We show that the equivalence classes of such abelian extensions are classified by the second cohomology group $H^2_{\mathrm{mRBA}} ( (A,R), (M,S))$.

\medskip

Let $(A, R)$ be a modified Rota-Baxter algebra of weight $\kappa$. Let $(M,S)$ be a pair consisting of a vector space $M$ and a linear map $S: M \rightarrow M$. Note that the pair $(M,S)$ can be regarded as a modified Rota-Baxter algebra of weight $\kappa$, where we equip $M$ with the trivial associative multiplication.

\begin{definition}
(i) An {\bf abelian extension} of the modified Rota-Baxter algebra $(A, R)$ of weight $\kappa$ by a pair $(M,S)$ is a short exact sequence of modified Rota-Baxter algebras
\begin{align}\label{diag}
\xymatrix{
0 \ar[r] & (M,S) \ar[r]^{i} & (E,U) \ar[r]^{\pi} & (A,R) \ar[r] & 0.
}
\end{align}
We often denote an abelian extension as above simply by $(E, U)$ when the structure maps $i$ and $\pi$ are clear from the context.

(ii) Two abelian extensions $(E,U)$ and $(E', {U'})$ of the modified Rota-Baxter algebra $(A, R)$ by a pair $(M,S)$ are said to be {\bf isomorphic} if there exists an isomorphism $\varphi : (E,U) \rightarrow (E', U')$ of modified Rota-Baxter algebras making the following diagram commutative
\begin{align}\label{diag2}
\xymatrix{
0 \ar[r] & (M,S) \ar[r]^{i} \ar@{=}[d] & (E,U) \ar[d]^{\varphi} \ar[r]^{\pi}  & (A,R) \ar[r] \ar@{=}[d] & 0 \\
0 \ar[r] & (M,S) \ar[r]_{i'} & (E,U) \ar[r]_{\pi'} & (A,R) \ar[r] & 0.
}
\end{align}
\end{definition}

Let $(E, U)$ be an abelian extension as of (\ref{diag}). A section of the map $\pi$ is a linear map $s : A \rightarrow E$ that satisfies $ \pi s = \mathrm{id}_A$. There is always a section of the map $\pi$. For any section $s$, we define two bilinear maps $A \times M \rightarrow M$, $(a, u) \mapsto s(a) \cdot_E i(u)$ and $M \times A \rightarrow M$, $(u,a) \mapsto i(u) \cdot_E s(a)$, for $a \in A$, $u \in M$. Here $~\cdot_E~$ denotes the associative multiplication on $E$. With these bilinear maps, it is easy to see that $M$ is an $A$-bimodule (see \cite{loday-cyclic}). More generally, the pair $(M,S)$ is a bimodule over the modified Rota-Baxter algebra $(A, R)$ of weight $\kappa$. This bimodule is called the induced bimodule over $(A,R)$.

Let $(M,S)$ be a given bimodule over a modified Rota-Baxter algebra $(A,R)$ of weight $\kappa$. We denote $\mathrm{Ext}((A,R), (M,S))$ by the set of all isomorphism classes of abelian extensions of $(A, R)$ by the pair $(M,S)$ for which the induced bimodule on $(M,S)$ is the given one.

\begin{theorem}
Let $(A, R)$ be a modified Rota-Baxter algebra of weight $\kappa$, and $(M,S)$ be a bimodule over it. Then there is a one-to-one correspondence between $\mathrm{Ext}((A,R), (M,S))$ and the second cohomology group $H^2_\mathrm{mRBA} ((A,R), (M,S))$.
\end{theorem}

\begin{proof}
Let $(\chi, \Phi) \in Z^2_\mathrm{mRBA} ((A, R), (M,S))$ be a $2$-cocycle that represents an element in $H^2_\mathrm{mRBA} ((A, R), (M,S))$. Since $(\chi, \Phi)$ is a $2$-cocycle, we have $\delta_\mathrm{Hoch} (\chi) = 0$ and $\widetilde{\delta}_\mathrm{Hoch} (\Phi) + \Psi^2 (\chi) = 0$. Consider the direct sum $A \oplus M$ with the multiplication 
\begin{align*}
(a, u) \cdot_\ltimes (b, v) = (a \cdot b, ~ a \cdot v + u \cdot a + \chi (a, b)), \text{ for } (a, u), (b, v) \in A \oplus M.
\end{align*}
Since $\chi$ is a Hochschild $2$-cocycle, it follows that the above multiplication is associative. Thus it makes the space $A \oplus M$ into an associative algebra, denoted by $A \oplus_\chi M$. Further, we define a map $R_\Phi : A \oplus_\chi M \rightarrow A \oplus_\chi M$ by
\begin{align*}
R_\Phi (a,u) = \big(  R(a), S (u) + \Phi (a) \big), \text{ for } (a, u) \in A \oplus M.
\end{align*}
Since $\widetilde{\delta}_\mathrm{Hoch} (\Phi) + \Psi^2 (\chi) = 0$, it follows that $R_\Phi$ is a modified Rota-Baxter operator of weight $\kappa$ on the algebra $A \oplus_\chi M$. In other words, $(A \oplus_\chi M, R_\Phi)$ is a modified Rota-Baxter algebra of weight $\kappa$. Then it is easy to see that
\begin{align*}
\xymatrix{
0 \ar[r] & (M,S) \ar[r]^{i} & (A \oplus_\chi M, R_\Phi) \ar[r]^{\pi} & (A,R) \ar[r] & 0
}
\end{align*}
is an abelian extension of $(A, R)$ by $(M,S)$, where the structure maps $i$ and $\pi$ are respectively given by $i (u) = (0, u)$ and $\pi (a, u) = a$, for $u \in M$, $(a, u) \in A \oplus M$. Note that the map $s : A \rightarrow A \oplus M$, $s(a) = (a, 0)$ is a section of the map $\pi$, and the section $s$ induces the given bimodule structure on $(M,S).$ 

\medskip

Let $(\chi', \Phi') \in Z^2_\mathrm{mRBA} ((A, R), (M,S))$ be another $2$-cocycle cohomologous to $(\chi, \Phi)$. In other words, they corresponds to the same element in $H^2_\mathrm{mRBA} ((A,R), (M,S))$. Then there exists a map $\theta \in \mathrm{Hom} (A, M)$ such that
\begin{align*}
(\chi, \Phi) - (\chi', \Phi') = \delta_\mathrm{mRBA} (\theta, 0).
\end{align*}
We now define a map $\varphi : A \oplus M \rightarrow A \oplus M$ by
\begin{align*}
\varphi (a, u) = (a, u - \theta (a)).
\end{align*}
It is then easy to verify that $\varphi : (A \oplus_\chi M, R_\Phi) \rightarrow (A \oplus_{\chi'} M, R_{\Phi'})$ is an isomorphism of modified Rota-Baxter algebras, which makes the abelian extensions $(A \oplus_\chi M, R_\Phi) $ and $(A \oplus_{\chi'} M, R_{\Phi'})$ isomorphic. Therefore, there exists a well-defined map $\Upsilon : H^2_\mathrm{mRBA} ((A, R), (M,S)) \rightarrow \mathrm{Ext}((A, R), (M, S))$.

\medskip

Conversely, let $(E, U)$ be an abelian extension of $(A, R)$ by $(M,S)$ so that the induced bimodule structure on $(M,S)$ coincides with the given one. Let $s : A \rightarrow E$ be any section of the map $\pi$. We define maps $\chi^s \in \mathrm{Hom}(A^{\otimes 2}, M)$ and $\Phi^s \in \mathrm{Hom}(A, M)$ by
\begin{align*}
\chi^s (a, b) = s(a) \cdot_E s(b) - s (a \cdot b) ~~~ \text{ and } ~~~ \Phi^s (a) = Us (a) - s R(a), \text{ for } a, b \in A.
\end{align*}
Then straightforward calculations show that $\chi^s$ is a Hochschild $2$-cocycle (i.e. $\delta_\mathrm{Hoch} (\psi^s) = 0$) and
\begin{align*}
\widetilde{\delta}_\mathrm{Hoch} (\Phi^s) + \Psi^2 (\chi^s)  = 0.
\end{align*}
Thus, we have $\delta_\mathrm{mRBA} (\chi^s, \Phi^s) = \big( \delta_\mathrm{Hoch} (\chi^s), ~ - \widetilde{\delta}_\mathrm{Hoch} (\Phi^s) - \Psi^2 (\chi^s)  \big)  = 0$, which shows that $(\chi^s, \Phi^s)$ is a $2$-cocycle. It is easy to see that the corresponding class in $H^2_\mathrm{mRBA} ((A, R), (M,S))$ doesn't depend on the choice of the section $s$.

\medskip

Let $(E,U)$ and $(E', U')$ be two isomorphic abelian extensions of as (\ref{diag2}). If $s: A \rightarrow E$ is a section of the map $\pi$, then the map $s' = \varphi s : A \rightarrow E'$ is a section of the map $\pi'$. Let $(\chi^{s'}, \Phi^{s'})$ be the $2$-cocycle corresponding to the abelian extension $(E', {U'})$ and the section $s'$ of the map $\pi'$. Then we have
\begin{align*}
\chi^{s'} (a, b) =~& s' (a) \cdot_{E'} s' (b) - s' (a \cdot b) \\
=~& (\varphi s)(a) \cdot_{E'} (\varphi s) (b) - (\varphi s)(a \cdot b) \\
=~& \varphi \big(  s(a) \cdot_E s(b) - s(a \cdot b) \big) = \varphi (\chi^s (a, b) ) = \chi^s (a, b)
\end{align*}
and
\begin{align*}
\Phi^{s'} (a) =~& Us' (a) - s' R(a) \\
=~& U \varphi s (a) - \varphi s R (a) \\
=~& \varphi Us (a) - \varphi s R(a) = \varphi (\Phi^s (a)) = \Phi^s (a).
\end{align*}
Hence $(\chi^{s'}, \Phi^{s'}) = (\chi^s, \Phi^s)$. Therefore, they corresponds to the same element in $H^2_\mathrm{mRBA} ((A, R), (M,S))$. As a summary, we obtain a well-defined map $\Omega : \mathrm{Ext}((A, R), (M,S)) \rightarrow H^2_\mathrm{mRBA} ((A, R), (M,S))$. Finally, the maps $\Upsilon$ and $\Omega$ are inverses to each other. This completes the proof.
\end{proof}

\medskip

We end this paper by suggesting some questions of further interest. We will address some of these questions in future.

We have seen that Rota-Baxter algebras are closely related to modified Rota-Baxter algebras of negative weights. In \cite{DasSK,wang-zhou} the authors constructs an $L_\infty$-algebra whose Maurer-Cartan elements are precisely Rota-Baxter algebra structures on a vector space $A$. One requires a suitable linear transformation of this $L_\infty$-algebra to obtain a new one whose Maurer-Cartan elements are precisely modified Rota-Baxter algebras of fixed negative weight. In a subsequent work, we aim to provide a systematic construction of the $L_\infty$-algebra that characterize modified Rota-Baxter algebras of any weight as Maurer-Cartan elements.

Given an algebraic structure, it is an interesting question to find the corresponding homotopy version. Such homotopy version is better understood in terms of the minimal model of the operad governing such structure. The minimal model of the operad for Rota-Baxter algebras was considered in \cite{wang-zhou}. It could be interesting to find the similar theory for modified Rota-Baxter algebras.

In \cite{modified-sheng} the authors considered modified Rota-Baxter operators of weight $-1$ on Lie algebras and studied their cohomology and deformation theory. One can easily generalize this to modified Rota-Baxter operators of arbitrary weight. Let $\mathfrak{g} = (\mathfrak{g}, [~,~])$ be a Lie algebra. A linear map $R : \mathfrak{g} \rightarrow \mathfrak{g}$ is a modified Rota-Baxter operator of weight $\kappa$ if 
\begin{align*}
[R(x), R(y) ] = R \big( [R(x), y] + [x, R(y)]  \big) + \kappa ~ \! [x, y], \text{ for } x, y \in \mathfrak{g}.
\end{align*}
It follows that if $(A, R)$ is a modified Rota-Baxter algebra of weight $\kappa$ then $R$ is a modified Rota-Baxter operator of the same weight on the commutator Lie algebra structure on $A$. We define a modified Rota-Baxter Lie algebra of weight $\kappa$ as a pair $(\mathfrak{g}, R)$ in which $\mathfrak{g}$ is a Lie algebra and $R: \mathfrak{g} \rightarrow \mathfrak{g}$ is a modified Rota-Baxter operator of weight $\kappa$. By following the results of the present paper, one could also develop the cohomology and deformation theory of modified Rota-Baxter Lie algebras of any weight.

In the study of integrations of Rota-Baxter Lie algebras, the authors in \cite{RBG} introduced Rota-Baxter operators of weight $1$ on abstract groups. Let $(G, \diamond)$ be a group. A set map $B : G \rightarrow G$ is a Rota-Baxter operator of weight $1$ if  
\begin{align*}
B(x) \diamond B(y) = B \big( x \diamond B(x) \diamond y \diamond B(x)^{-1} \big), \text{ for } x, y \in \mathfrak{g}.
\end{align*} 
An interesting problem is to find the suitable notion of modified Rota-Baxter operator of weight $-1$ on abstract groups. With this notion, one could expect integrations of modified Rota-Baxter Lie algebras of weight $-1$ by modified Rota-Baxter Lie groups of weight $-1$.

\medskip

\medskip

\noindent {\bf Acknowledgements.} The author would like to thank IIT Kharagpur (India) for providing a beautiful academic atmosphere where his research has been carried out.


\begin{thebibliography}{BFGM03}


%\bibitem{aguiar-inf} M. Aguiar, Infinitesimal Hopf algebras, {Contemporary Mathematics, New trends in Hopf algebra theory (La Falda, 1999), 1--29,} { Contemp. Math.,} 267, Amer. Math. Soc., Providence, RI, (2000).

\bibitem{aguiar} M. Aguiar, Pre-Poisson algebras, {\em Lett. Math. Phys.} 54 (2000) 263-277.

%\bibitem{aguiar-bial} M. Aguiar, On the associative analog of Lie bialgebras, {J. Algebra} 244 (2001) 492--532.


\bibitem{aguiar-pre-lie} M. Aguiar, Infinitesimal bialgebras, pre-Lie and dendriform algebras, In {\em Hopf algebras}, 1-33, Lecture Notes in Pure and Applied Mathematics 237. New York: Dekker, 2004.

%\bibitem{aguiar-loday}
%M. Aguiar and J.-L. Loday, Quadri-algebras,
%{\em J. Pure Appl. Algebra} 191 (2004), no. 3, 205-221.

%\bibitem{alek} D. Alekseevsky, P. W. Michor and W. Ruppert, Extensions of Lie algebras, arXiv:math.DG/0005042.

%\bibitem{alek2} D. Alekseevsky, P. W. Michor and W. Ruppert, Extensions of super Lie algebras, {\em J. Lie Theory} 15 (2005) No. 1, 125-134.

\bibitem{atkinson} Atkinson, F.V.: Some aspects of Baxter's functional equation. {J. Math. Anal. Appl.} 7, 1--30 (1963).

%\bibitem{baez-crans} J. C. Baez and A. S. Crans, Higher-Dimensional Algebra VI: Lie $2$-algebras, {\em Theor. Appl. Categ.} 12 (2004) 492-538.

%\bibitem{bai1} C. Bai, Double constructions of Frobenius algebras, Connes cocycles and their duality, { J. Noncommut. Geom.,} 4 (2010)  475--530.

\bibitem{bai-lax} C. Bai, L. Guo and X. Ni, Non-abelian generalized Lax pairs, the classical Yang-Baxter equation and
Post-Lie algebras, {\em Commun. Math. Phys.} 297 (2010), 553-596. 

\bibitem{bai} C. Bai, O. Bellier and L. Guo, Splitting of operations, Manin products, and Rota-Baxter operators, {\em International Mathematics Research Notices} 2013 (3) (2013) 485--524.

%\bibitem{bai-o} Bai, C., Guo, L., Ni, X.: $\mathcal{O}$-operators on associative algebras and dendriform algebras. { J. Algebra Appl.} 12, 1350027 (2013).

\bibitem{balav} D. Balavoine, D\'{e}formations des alg\'{e}bres de Leibniz (in French), {\em C. R. Acad. Sci. Paris S\'{e}r. I Math.} 319 (1994), no. 8, 783-788.

\bibitem{bala} D. Balavoine, Deformations of algebras over a quadratic operad, {Operads: Proceedings of Renaissance Conferences (Hartford, CT/Luminy, 1995),} 207--234, Contemp. Math., 202, { Amer. Math. Soc., Providence, RI,} 1997.

%\bibitem{bar-singh} V. G. Bardakov and M. Singh, Extensions and automorphisms of Lie algebras, {\em J. Algebra Appl.} 16 (2017) 1750162, 15pp. 

\bibitem{baxter} G. Baxter, An analytic problem whose solution follows from a simple algebraic identity, {\em Pacific J. Math.} 10  (1960) 731--742.

%\bibitem{beidar} K. I. Beidar, Y. Fong and A. Stolin, On Frobenius algebras and the quantum Yang-Baxter equation, {\em Trans. Amer. Math. Soc.} 349 (1997), 3823-3836.

\bibitem{bor} M. Bordemann, Generalized Lax pairs, the modified classical Yang–Baxter equation, and affine geometry of Lie groups, {\em Commun. Math. Phys.} 135 (1990), 201–216. 

%\bibitem{braun} C. Braun, Involutive $A_\infty$-algebras and dihedral cohomology, {\em J. Homotopy Relat. Struct.} 9 (2014) 317-337.

%\bibitem{brz} T. Brzezi\'{n}ski, Rota-Baxter systems, dendriform algebras and covariant bialgebras, {\em J. Algebra} 460 (2016) 1-25.

\bibitem{cartier} P. Cartier, On the structure of free Baxter algebras. Adv. Math. 9, 253--265 (1972).

%\bibitem{grab} 
%J. F. Cari\~{n}ena, J. Grabowski and G. Marmo, Quantum Bi-Hamiltonian Systems, {\em Internat. J. Modern. Phys. A} 15 (2000) 4797-4810.

%\bibitem{casas} J. M. Casas, E. Khmaladze and M. Ladra, Low-dimensional non-abelian Leibniz cohomology, {\em Forum Math.} 25 (2013), no. 3, 443-469.

%\bibitem{Chapoton} F. Chapoton, Un th\'eor\`eme de Cartier-Milnor-Moore-Quillen pour les big\`ebres dendrformes et les alg\`ebres braces, J. Pure Appl. Algebra 168 (2002) 1--18.

%\bibitem{chap-mur} F. Chapoton and M. Livernet, Pre-Lie algebras and the rooted trees operad, {\em Internat. Math. Res. Notices} 2001 (8) (2001) 395--408.

%\bibitem{conn94} A. Connes, Noncommutative Geometry, Academic Press, San Diego, (1994).

\bibitem{connes}
A. Connes and D. Kreimer, Renormalization in quantum field theory and the Riemann-Hilbert problem. I. The Hopf algebra structure of graphs and the main theorem,
{\em Comm. Math. Phys.} 210 (1)   (2000) 249--273.

%\bibitem{costello} K. Costello, Topological conformal field theories and Calabi-Yau categories, {\em Adv. Math.} 210 (2007) 165-214.

\bibitem{das-rota} A. Das, Deformations of associative Rota-Baxter operators, {\em J. Algebra} 560 (2020) 144--180.

%\bibitem{das-dend} A. Das, Cohomology and deformations of dendriform algebras, and ${Dend}_\infty$-algebras. {\em Comm. Algebra}, Vol. 50 Issue 4 (2022) 1544-1567.

\bibitem{DasSK} A. Das and S. K. Mishra, The $L_{\infty}$-deformations of associative Rota-Baxter algebras and homotopy Rota-Baxter operators, {\em J. Math. Phys.} Vol. 63, Issue 5 (2022) 051703.

%\bibitem{das-mishra-hazra} A. Das, S. K. Mishra and S. K. Hazra, Non-abelian extensions of Rota-Baxter Lie algebras and inducibility of automorphisms, communicated.




%\bibitem{das-jhrs} A. Das, Leibniz algebras with derivations, {\em J. Homotopy Relat. Struct.} 16 (2021) 245-274.

%\bibitem{DasSK2} A. Das and S. K. Mishra, Bimodules over relative Rota-Baxter algebras and cohomologies, communicated.

%\bibitem{Das-nis} A. Das and N. Rathee, Extensions of Rota-Baxter groups, in preparation.

%\bibitem{das3} A. Das, Deformations of Loday-type algebras and their morphisms, {\em J. Pure Appl. Algebra} to appear

%\bibitem{das3} A. Das, Compatible $\mathcal{O}$-operators on modules over Lie algebras,

%\bibitem{das-mandal} A. Das and A. Mandal, Extensions, deformations and categorifications of AssDer pairs,

%\bibitem{das-saha} A. Das and R. Saha, Involutive and oriented dendriform algebras, arXiv:2006.01483, {\em submitted for publication}

%\bibitem{das-system} A. Das, Generalized Rota-Baxter systems, arXiv:2007.13652, {\em submitted for publication}

%\bibitem{dot-rota} V. Dotsenko and A. Khoroshkin, Quillen homology for operads via Gr\"{o}bner bases, {em Doc. Math.} 18 (2013), 707-747.

\bibitem{doubek} M. Doubek, M. Markl and P. Zima, Deformation theory (lecture notes), {\em Arch. Math. (Brno)} 43 (5) (2007) 331-371.

%\bibitem{perturb1} V. G. Drinfel'd, On quasitriangular quasi-Hopf algebras and a certain group closely connected with $\mathrm{Gal}( \overline{\mathbb{Q}}/\mathbb{Q})$, {\em Leningrad Math. J.} 2 (1991) 829-860.

%\bibitem{perturb2} V. G. Drinfel'd, On the structure of quasitriangular quasi-Hopf algebras, {\em Func. Anal. Appl.} 26 (1992) 63-65.

%\bibitem{eilen} S. Eilenberg and S. Maclane, Cohomology theory in abstract groups. II. Group extensions with non-abelian kernel, {\em Ann. of Math.} 48 (1947) 326-341. 

%\bibitem{fer} R. Fern\'{a}ndez-Val\'{e}ncia and J. Giansiracusa, On the Hochschild homology of involutive algebras, {\em Glasg. Math. J.} 60 (2018) 187-198.

%\bibitem{fial} A. Fialowski, Construction of miniversal deformations of Lie algebras, {J. Funct. Anal.} 161 (1) (1999) 76--110.

%\bibitem{fial-pen} A. Fialowski and M. Penkava, Extensions of (super) Lie algebras, {\em Commun. Contemp. Math.} 11 (2009), no. 5, 709-737.

%\bibitem{yael} Y. Fr\'{e}gier, Non-abelian cohomology of extensions of Lie algebras as Deligne groupoid, {\em J. Algebra} 398 (2014) 243-257.

%\bibitem{yael} Y. Fr\'{e}gier, M. Markl and D. Yau, The $L_\infty$-deformation complex of diagrams of algebras, {\em New York J. Math.} 15 (2009) 353-392.

%\bibitem{yael-zam} Y. Fr\'{e}gier and M. Zambon, Simultaneous deformations of algebras and morphisms via derived brackets, {\em J. Pure Appl. Algebra} 219 (2015), 5344--5362.

\bibitem{gers-ring} M. Gerstenhaber, The cohomology structure of an associative ring, {\em Ann. of Math.} (2) 78 (1963), 267-288.

\bibitem{gers} M. Gerstenhaber, On the deformation of rings and algebras, {\em Ann. of Math. (2)} 79 (1964), 59--103.

%\bibitem{gers-sch} M. Gerstenhaber and S. D. Schack, On the deformation of algebra morphisms and diagrams, {\em Trans. Amer. Math. Soc.} 279 (1983), no. 1, 1-50.

%\bibitem{gers-voro}
%M. Gerstenhaber and A. A. Voronov, Homotopy $G$-algebras and moduli space operad, {\em Internat. Math. Res. Notices} 1995, no. 3, 141-153.

\bibitem{gers-sch} M. Gerstenhaber and S. D. Schack, On the deformation of algebra morphisms and diagrams, {\em Trans. Amer. Math. Soc.} 279 (1983), no. 1, 1-50.

%\bibitem{getzler} E. Getzler, Lie theory for nilpotent $L_\infty$-algebras, {Ann. of Math. (2)} 170 (1) (2009) 271--301.

%\bibitem{gon-kol} Goncharov, M.E., Kolesnikov, P.S.: Simple finite-dimensional double algebras. {J. Algebra} 500, 425--438 (2018).

%\bibitem{gouray} J.-B. Gouray, A differential graded Lie algebra approach to non abelian extensions of associative algebras, arXiv preprint, arXiv:1802.04641

\bibitem{guo-book} L. Guo, An introduction to Rota-Baxter algebra. Surveys of Modern Mathematics, 4. { International Press, Somerville, MA; Higher Education Press, Beijing} (2012).

\bibitem{guo-keigher} L. Guo and W. Keigher, Baxter algebras and Shuffle products, {\em Adv. Math.} 150 (2000) 117-149.

\bibitem{GuoLin} L. Guo and Z. Lin, Representations and modules of Rota-Baxter algebras, Preprint, arXiv:1905.01531

%\bibitem{guo-keigher} L. Guo and W. Keigher, On differential Rota-Baxter algebras, {\em J. Pure Appl. Algebra} 212 (2008), no. 3, 522-540.

\bibitem{RBG} L. Guo, H. Lang and Y. Sheng, Integration and geometrization of Rota-Baxter Lie algebras, {\em Adv. Math.} 387 (2021) 107834.

%\bibitem{hazra-habib} S. K. Hazra and A. Habib, Wells exact sequence and automorphisms of extensions of Lie superalgebras, {\em J. Lie Theory} 30 (2020) 179-199.

\bibitem{hoch} G. Hochschild, On the cohomology groups of an associative algebra, {\em Ann. of Math. (2)} 46 (1945), 58-67.

%\bibitem{hochschild} G. Hochschild and J.-P. Serre, Cohomology of group extensions, {\em Trans. Amer. Math. Soc.} 74 (1953) 110-134.

%\bibitem{hoch} G. Hochschild, Cohomology classes of finite type and finite dimensional kernels for Lie algebras, {\em Am. J. Math.} 76 (1954) 763-778.

%\bibitem{inas} N. Inassaridze, E. Khmaladze and M. Ladra, Non-abelian cohomology and extensions of Lie algebras, {\em J. Lie Theory} 18 (2008) 413-432.


\bibitem{jiang-sheng} J. Jiang and Y. Sheng, Representations and cohomologies of relative Rota-Baxter Lie algebras and applications, {\em J. Algebra} 602 (2022), 637-670.

\bibitem{modified-sheng} J. Jiang and Y. Sheng, Cohomologies and deformations of modified $r$-matrices, Preprint, arXiv:2206.00411

%\bibitem{jin} P. Jin and H. Liu, The Wells exact sequence for the automorphism group of a group extension, {\em J. Algebra} 324 (2010) 1219-1228.

%\bibitem{joni-rota} S. A. Joni and G. C. Rota, Coalgebras and Bialgebras in combinatorics, {Stud. Appl. Math.} 61 (1979), no.2, 93-139. Reprinted in {\em Gian-Carlo Rota on Combinatorics: Introductory papers ans commentaries} (Joseph P.S. Kung, Ed.), Birkh\"{a}user, Boston (1995).

%\bibitem{keller} Keller, B.: Introduction to $A$-infinity algebras and modules. {Homology Homotopy Appl.} 3(1), 1--35 (2001).

%\bibitem{kont-soi} M. Kontsevich and Y. Soibelman, Deformation theory I. Book in preparation

%\bibitem{kuper} B. A. Kupershmidt, What a classical $r$-matrix really is, {\em J. Nonlinear Math. Phys.} 6 (1999) 448--488.

%\bibitem{lada-markl} T. Lada and M. Markl, Strongly homotopy Lie algebras, {Comm. Algebra} 23 (6) (1995) 2147--2161.

%\bibitem{lada-stasheff} T. Lada and J. Stasheff, Introduction to SH Lie algebras for Physicists, {Internat. J. Theoret. Phys.} 32 (7) (1993) 1087--1103.

\bibitem{laza-rota} A. Lazarev, Y. Sheng and R. Tang, Deformations and homotopy theory of relative Rota-Baxter Lie algebras, {\em Comm. Math. Phys.} 385, 595--631 (2021).

\bibitem{li} L. C. Li, Classical $r$-matrices and compatible Poisson structures for Lax equations on Poisson algebras, {\em Comm. Math. Phys.} 203 (1999), no. 3, 573-592.

%\bibitem{liu-sheng-wang} J. Liu, Y. Sheng and Q. Wang, On non-abelian extensions of Leibniz algebras, {\em Comm. Algebra} 46 (2) (2018) 574-587.

%\bibitem{loday} Loday, J.L.: Dialgebras and related operads. pp. 7--66, Lecture Notes in Math., 1763, {em Springer, Berlin} (2001).

\bibitem{loday-cyclic} J.-L. Loday, Cyclic Homology. Springer-Verlag, Grundlehren der mathematischen Wissenschaften, 301 (1992).

%\bibitem{loday-der} J.-L. Loday, On the operad of associative algebras with derivation, {\em Georgian Math. J.} 17 (2010) 347-372.

%\bibitem{lod-val-book} J. L. Loday and B. Vallette, Algebraic operads,  Grundlehren der Mathematischen Wissenschaften, Springer, Heidelberg, 346 (2012).

%\bibitem{lue} A. Lue, Crossed homomorphisms of Lie algebras, {\em Proc. Cambridge Philos. Soc.} 62 (1966), 577-581.

%\bibitem{lyndon} R. C. Lyndon, The cohomology theory of group extensions, {\em Duke Math. J.} 15 (1948) 271–292.

%\bibitem{ma-et} T. Ma, A. Makhlouf and S. Silvestrov, Curved $\mathcal{O}$-operator systems, {\em arXiv:1710.05232}

%\bibitem{mars-ratiu} J. E. Marsden and T. Ratiu, Reduction of Poisson manifolds, {\em Lett. Math. Phys.} 11 (1986), no. 2, 161-169.

\bibitem{nij-ric} A. Nijenhuis and Richardson, Cohomology and deformations in graded Lie algebras, {Bull. Amer. Math. Soc.} 72 (1966) 1--29.

%\bibitem{passi} I. B. S. Passi, M. Singh and M. K. Yadav, Automorphisms of abelian group extensions, {\em J. Algebra} 324 (2010) 820-830.

%\bibitem{pei-guo} J. Pei and L. Guo, Averaging algebras, Schr\"{o}der numbers, rooted trees and operads,  {\em J. Algebraic Combin.} 42 (2015), no. 1, 73-109.

%\bibitem{pei} Y. Pei, Y. Sheng, R. Tang and K. Zhao, Generalized Shen-Larsson bifunctors and cohomologies of crossed homomorphisms, 

%\bibitem{qiu-chen} J. Qiu and Y. Chen, Free Rota-Baxter systems and a Hopf algebra structure, {\em Comm. Algebra} 46 (2018) 3913-3925.

%\bibitem{rey}
%O. Reynolds, On the dynamic theory of incompressible viscous fluids, {\em Phil. Trans. Roy. Soc.} A 136 (1895), 123-164.

%\bibitem{robinson}  D. J. S. Robinson, Automorphisms of group extensions, {\em Note Mat.} 33 (2013) 121-129.

\bibitem{rota} G.-C. Rota, Baxter algebras and combinatorial identities I, II. {\em Bull. Amer. Math. Soc.} 75, 325--329 (1969); ibid 75, 330--334 (1969).

\bibitem{semenov} M. A. Semenov-Tyan-Shanski\u{i}, What a classical $r$-matrix is. {\em Functional Anal. Appl.} 17 (1983), no. 4, 259-272.

%\bibitem{serre} J.-P. Serre, Cohomologie des extensions de groupes, {\em C. R. Acad. Sci. Paris} 231 (1950) 643-646.

%\bibitem{sev-wein} P. \v{S}evera and A. Weinstein, Poisson geometry with a $3$-form background, {\em Progr. Theoret. Phys. Suppl.} No. 144 (2001), 145-154.

%\bibitem{stas} Stasheff, J.: Homotopy associativity of $H$-spaces II. {Trans. Amer. Math. Soc.} 108, 293--312 (1963).

%\bibitem{stasheff} J. Stasheff, The intrinsic bracket on the deformation complex of an associative algebra, {  J. Pure Appl. Algebra} 89 (1993), 231--235.

\bibitem{sza} B. Szablikowski, Classical $r$-matrix like approach to Frobenius manifolds, WDVV equations and flat metrics,
{\em J. Phys. A} 48 (2015), no. 31, 315203, 47 pp.

\bibitem{tang} R. Tang,  C. Bai, L. Guo and Y. Sheng, Deformations and their controlling cohomologies of $\mathcal{O}$-operators, {\em Comm. Math. Phys.} 368 (2) (2019)  665--700.

%\bibitem{lieder} R. Tang, Y. Fr\'{e}gier and Y. Sheng, Cohomologies of a Lie algebra with a derivation and applications, {\em J. Algebra} 534 (2019) 65-99.

%\bibitem{uchino} Uchino, K.: Quantum analogy of Poisson geometry, related dendriform algebras and Rota-Baxter operators. {Lett. Math. Phys.} 85(2-3),  91--109 (2008).

%\bibitem{uchino2} K. Uchino, Twisting on associative algebras and Rota-Baxter type operators, {\em J. Noncommut. Geom.} 4 (3) (2010) 349--379.

%\bibitem{voro} Th. Voronov, Higher derived brackets and homotopy algebras, {\em J. Pure Appl. Algebra} 202 (1-3) (2005)133--153.

%\bibitem{voro2} Th. Voronov, Higher derived brackets for arbitrary derivations, {Travaux math\'{e}matiques. Fasc. XVI,} 163--186, Trav. Math., 16, Univ. Luxemb., Luxembourg, 2005.

\bibitem{wang-zhou} K. Wang and G. Zhou, Deformations and homotopy theory of Rota-Baxter algebras of any weight, Preprint, {arXiv:2108.06744}.

%\bibitem{weibel} C. A. Weibel, An introduction to homological algebra, Cambridge Studies in Advanced Mathematics, 38. {\em Cambridge University Press, Cambridge}, 1994.

%\bibitem{wells} C. Wells, Automorphisms of group extensions, {\em Trans. Amer. Math. Soc.} 155 (1971) 189-194.

%\bibitem{guo-quasi} H. Yu, L. Guo and J.-Y. Thibon, Weak quasi-symmetric functions, Rota-Baxter algebras and Hopf algebras, { Adv. Math.} 344 (2019) 1--34.

%\bibitem{zgg} Zhang, T., Gao, X., Guo, L.: Hopf algebras of rooted forests, cocycles, and free Rota-Baxter algebras. {J. Math. Phys.} 57, 101701 (2016).

\bibitem{zhang} X. Zhang, X. Gao and L. Guo, Modified Rota-Baxter Algebras, Shuffle Products and Hopf Algebras, {\em Bull. Malays. Math. Sci. Soc.} 42 (2019), 3047-3072.

\bibitem{zhang2} X. Zhang, X. Gao and L. Guo, Free modified Rota-Baxter algebras and Hopf algebras, {\em Int. Electron. J. Algebra} 25 (2019), 12-34.

\bibitem{zhang3} Z. Zhu, H. Zhang and X. Gao, Free weighted (modified) differential algebras, free (modified) Rota-Baxter algebras and Gröbner-Shirshov bases, Preprint, arXiv:2108.03563.

\end{thebibliography}
\end{document}